\documentclass[hidelinks,12pt,a4paper]{article}
\usepackage[utf8x]{inputenc}
\usepackage[top=30mm, bottom=30mm, inner=20mm, outer=20mm, headsep=10mm, footskip=12mm]{geometry}
\usepackage{amssymb}
\usepackage{amsmath,bm}
\usepackage{mathtools}
\usepackage{amsthm}
\usepackage[english]{babel}
\usepackage{lmodern}
\usepackage{microtype}
\usepackage[mathscr]{euscript}
\usepackage[colorlinks=true]{hyperref}
\usepackage[usenames,dvipsnames]{xcolor}
\usepackage{todonotes}
\usepackage{appendix}
\usepackage{bm}\usepackage{cancel}
\usepackage{upgreek}
\usepackage[normalem]{ulem}

\numberwithin{equation}{section}

\theoremstyle{plain}
\newtheorem{definition}{Definition}[section]

\newtheorem{theorem}[definition]{Theorem}
\newtheorem{proposition}[definition]{Proposition}
\newtheorem{lemma}[definition]{Lemma}
\newtheorem{corollary}[definition]{Corollary}

\theoremstyle{definition}
\newtheorem{remark}[definition]{Remark}

\newtheorem*{ack}{Acknowledgements}

\makeatletter
\let\save@mathaccent\mathaccent
\newcommand*\if@single[3]{%
  \setbox0\hbox{${\mathaccent"0362{#1}}^H$}%
  \setbox2\hbox{${\mathaccent"0362{\kern0pt#1}}^H$}%
  \ifdim\ht0=\ht2 #3\else #2\fi
  }
\newcommand*\rel@kern[1]{\kern#1\dimexpr\macc@kerna}
\newcommand*\widebar[1]{\@ifnextchar^{{\wide@bar{#1}{0}}}{\wide@bar{#1}{1}}}
\newcommand*\wide@bar[2]{\if@single{#1}{\wide@bar@{#1}{#2}{1}}{\wide@bar@{#1}{#2}{2}}}
\newcommand*\wide@bar@[3]{%
  \begingroup
  \def\mathaccent##1##2{%
    \let\mathaccent\save@mathaccent
    \if#32 \let\macc@nucleus\first@char \fi
    \setbox\z@\hbox{$\macc@style{\macc@nucleus}_{}$}%
    \setbox\tw@\hbox{$\macc@style{\macc@nucleus}{}_{}$}%
    \dimen@\wd\tw@
    \advance\dimen@-\wd\z@
    \divide\dimen@ 3
    \@tempdima\wd\tw@
    \advance\@tempdima-\scriptspace
    \divide\@tempdima 10
    \advance\dimen@-\@tempdima
    \ifdim\dimen@>\z@ \dimen@0pt\fi
    \rel@kern{0.6}\kern-\dimen@
    \if#31
      \overline{\rel@kern{-0.6}\kern\dimen@\macc@nucleus\rel@kern{0.4}\kern\dimen@}%
      \advance\dimen@0.4\dimexpr\macc@kerna
      \let\final@kern#2%
      \ifdim\dimen@<\z@ \let\final@kern1\fi
      \if\final@kern1 \kern-\dimen@\fi
    \else
      \overline{\rel@kern{-0.6}\kern\dimen@#1}%
    \fi
  }%
  \macc@depth\@ne
  \let\math@bgroup\@empty \let\math@egroup\macc@set@skewchar
  \mathsurround\z@ \frozen@everymath{\mathgroup\macc@group\relax}%
  \macc@set@skewchar\relax
  \let\mathaccentV\macc@nested@a
  \if#31
    \macc@nested@a\relax111{#1}%
  \else
    \def\gobble@till@marker##1\endmarker{}%
    \futurelet\first@char\gobble@till@marker#1\endmarker
    \ifcat\noexpand\first@char A\else
      \def\first@char{}%
    \fi
    \macc@nested@a\relax111{\first@char}%
  \fi
  \endgroup
}
\makeatother

\renewcommand{\d}{\,\mathrm{d}}

\usepackage[xcolor]{changebar}
\cbcolor{blue}

\newcommand{\Int}{\mathrm{Int}}
\DeclareMathOperator{\Cone}{Cone}

\usepackage[inline]{enumitem}
\newcommand{\enumlabelformat}{\roman}

\newlength{\thelabelsep}
\newcounter{inlineenum}
\renewcommand{\theinlineenum}{\enumlabelformat{inlineenum}}

\let\epsilon\varepsilon
\let\phi\varphi

\DeclareMathOperator{\arcosh}{arcosh}

\newcommand{\nchi}{{\raise.3ex\hbox{$\chi$}}}

\usepackage[runin]{abstract}

\newcommand{\lm}[1]{\mathbb{L}^2(#1)}

\newcommand{\LpLS}{Lo\-rentz\-ian pre-length space }
\newcommand{\LpLSn}{Lo\-rentz\-ian pre-length space}

\newcommand{\LpLSsn}{Lo\-rentz\-ian pre-length spaces}

\newcommand*{\bx}{\bar{x}}
\newcommand*{\by}{\bar{y}}
\newcommand*{\bz}{\bar{z}}

\newcommand*{\bw}{\bar{w}}
\newcommand*{\bp}{\bar{p}}
\newcommand*{\bq}{\bar{q}}

\newcommand{\hx}{\hat{x}}
\newcommand{\hy}{\hat{y}}
\newcommand{\hz}{\hat{z}}
\newcommand{\hp}{\hat{p}}
\newcommand{\hhm}{\hat{m}}

\newcommand{\hw}{\hat{w}}

\newcommand{\ma}{\ensuremath{\measuredangle}}

\DeclareMathOperator{\TSec}{TSec}

\renewcommand{\d}{\,\mathrm{d}}

\makeatletter
\let\@fnsymbol\@arabic
\makeatother

\allowdisplaybreaks

\title{Lorentz meets Ptolemy}
\author{Felix Rott\footnotemark[1], Zhe-Feng Xu\footnotemark[1],\textsuperscript{,}\footnotemark[2] \,\,Matteo Zanardini\footnotemark[1]}
\date{\today}

\begin{document}

\maketitle
\footnotetext[1]{SISSA, 34136, Trieste, Italy, frott@sissa.it, zxu@sissa.it, mzanardi@sissa.it} \footnotetext[2]{School of Mathematical Sciences, University of Science and Technology of China,  230026, Hefei, China, xzf1998@mail.ustc.edu.cn}
\begin{abstract}
We consider a Lorentzian analogue of the Ptolemy inequality and we prove that in the setting of globally hyperbolic spacetimes it is equivalent to a global timelike sectional curvature bound from above by zero. We investigate the link between the Ptolemy inequality and the hyperbolic inversion and establish some applications and rigidity properties. 
\bigskip

\medskip

\end{abstract}
\tableofcontents

\section{Introduction} 
Famously, the Ptolemy inequality says that any quadruple of points in the Euclidean plane satisfies the following inequality: 
\begin{equation}
\label{eq: OG Ptolemy}
\widebar{AB} \cdot \widebar{CD} + \widebar{BC} \cdot \widebar{AD} \geq \widebar{AC} \cdot \widebar{BD} \, .
\end{equation}
Moreover, equality holds if and only if the four points lie on a circle or on a line. 
It is named after the famous Greek astronomer and  mathematician Claudius Ptolemy, who mentioned it in his work \emph{Almagest} roughly two thousand years ago. 
The inequality can be interpreted as a relation between the product of the diagonals of a quadrilateral and the sum of the products of opposite sides. 
The inequality is also valid in higher dimensions. 

In a more modern context and outside of Euclidean space, the Ptolemy inequality at least goes back to being mentioned in \cite{Sch40}. 
The inequality is more prominently featured as a characterisation of inner product spaces among merely normed ones, much in the spirit of the parallelogram identity \cite{Sch52}. 

Even more recently, there are a number of results about Ptolemaic metric spaces, meaning that \eqref{eq: OG Ptolemy} is required to hold for any quadruple of points in a metric space. 
The Ptolemy inequality has a strong connection to (sectional) curvature bounds from above. 
Among other things, on Riemannian manifolds, the Ptolemy inequality is equivalent to the space being CAT(0), and Finsler manifolds enjoying this property are necessarily Riemannian \cite{BuFaWra}. 
It is further known that on Ptolemaic metric spaces, the notions of being CAT(0) and Busemann convex are equivalent, where in general CAT(0) is strictly stronger than Busemann convexity \cite{FLS07}. 

To the best of our investigative abilities, we have found no mention of the Ptolemy inequality in a Lorentzian context. 
We propose a definition for (non-smooth) Ptolemaic Lorentzian spaces and collect various results. 
In particular, we show that Minkowski space is Ptolemaic. 
To this end, much like in positive signature, our approach makes use of inversion techniques. 
In Lorentzian signature, this idea was already mentioned in \cite{Chr} in the particular case of Minkowski space. More related to physics, the class of symmetries generated by inversions plays an important role in conformal field theories; see e.g.\ \cite{CFTbook}. 

We extend this concept of inversion to more general Lorentzian spaces and moreover, in the spirit of \cite{BHX08}, give a definition of an inversion time separation function. 
Our main theorem concerns the interplay between the Ptolemy inequality and a bound from above on timelike sectional curvature. 
\begin{theorem}
A globally hyperbolic spacetime is Ptolemaic if and only if it is globally non-positively curved.
\end{theorem}
Finally, we will use some rigidity properties of the hyperbolic inversion to single out spaces of vanishing timelike sectional curvature and give a characterization of Minkowski space among spacetimes.

\section{Preliminaries}
In this section we recall some basic information regarding smooth and non-smooth Lorentzian geometry. 
Let us start with the smooth case. 
For further details we refer to \cite{BeemBook, Min19}. 

A spacetime is a non-compact smooth manifold $M$ endowed with a Lorentzian metric tensor $g$ with signature $(+, -, \dots -)$ and time-oriented by a timelike vector field $X \in \mathfrak{X}(M)$. 
Given a point $p \in M$, we say that a vector $v \in T_pM$ is timelike provided $g_p(v,v) >0$ and we say that it is null provided $g_p(v,v) =0$. 
Otherwise, we call $v$ spacelike.
We call $v \in T_pM$ causal if it is null or timelike and, moreover, we say that $v$ is future-directed (resp.\ past-directed) provided $g(v,X(p)) \geq 0$ (resp.\ $\leq 0$). 
A curve is called future-directed timelike (resp.\ causal) if its tangent vector at each point is future-directed timelike (resp.\ causal). 
The class of future-directed causal vectors in $T_pM$ is denoted by $F_pM$, and we define the hyperbolic norm out of $g$ on the tangent space $T_pM$ by
\begin{equation*}
\| v\|_{F_g} : = \begin{cases}
\sqrt{g_p(v,v)} \, , & \text{if } v \in F_pM, \\
0 \,, & \text{otherwise} \, .
\end{cases}
\end{equation*} 
We refer to \cite[Section A.3]{Octet} and \cite{HBS} for an introduction to hyperbolic (Banach) norms. 
Notice that, in our framework, we define $\|v\|_{F_g}=0$ (instead of $- \infty$) if $v$ is spacelike to have an immediate compatibility with the current literature on triangle comparison in Lorentzian pre-length spaces.

We define a two-variable function that acts dually to the distance induced by a Riemannian metric: this is the time separation function $\ell: M \times M \rightarrow [0, + \infty]$ and it is given by the formula
\begin{equation}
    \ell(x,y) := \sup \int_0^1 \| \dot \gamma_t \|_{F_g} \d t \, ,
\end{equation}
where the supremum is taken over all piecewise $C^1$ future-directed causal curves $\gamma :[0,1] \rightarrow M$ such that $\gamma_0=x$ and $\gamma_1=y$. 
If there are no such curves, by convention we set $\ell(x,y)= 0$. 
We define the causal (resp.\ timelike) relations $\leq$ (resp.\ $\ll$)  on $M$ by saying that two points are related if they can be connected by a future-directed causal (resp.\ timelike) curve. 
Moreover, we use the notation
\[
y \in J^+(x) \, \Leftrightarrow \, x \leq y \, , \qquad \qquad y \in I^+(x) \, \Leftrightarrow \, x \ll y \, ,
\]
and similarly for the past versions of these sets. 

\begin{definition}[Strong causality, global hyperbolicity and future one-connectedness]
Let $(M,g)$ be a spacetime. 
\begin{itemize}
\item[(i)] We say that $M$ is strongly causal if the timelike diamonds $I(x,y):=I^+(x) \cap I^-(y)$ form a subbase for the topology.  
\item[(ii)] We say that it is globally hyperbolic provided $\leq$ is a partial order on $M$ and the causal diamonds $J(x,y) : = J^+(x) \cap J^-(y)$ are compact in $M$.
\item[(iii)] We say that it is future one-connected if there exists a timelike homotopy between any couple of future-directed timelike curves with the same endpoints. That is to say, for every couple of future-directed timelike smooth curves $\gamma, \sigma :[0,1] \rightarrow M$ such that $\gamma_0= \sigma_0$ and $\gamma_1=\sigma_1$ there exists a smooth map 
\[
\Gamma: [0,1] \times [0,1] \rightarrow M
\]
such that $\Gamma (0, \cdot)= \gamma$, $\Gamma(1, \cdot)= \sigma$ and $\Gamma(t, \cdot)$ is a timelike curve connecting $\gamma_0$ to $\gamma_1$.
\end{itemize}
\end{definition} 

We remark that on a globally hyperbolic spacetime, $\ell$ is finite and continuous \cite[Theorem 4.124]{Min19}, and between any pair of causally related points there exists a maximising geodesic connecting them \cite[Theorem 4.123]{Min19}. 
On a spacetime $(M,g)$, the Riemann tensor is defined in the usual way via the Levi-Civita connection by the formula
\begin{align}
R(X,Y)Z:=\nabla_X \nabla_Y Z - \nabla_Y \nabla_X Z - \nabla_{[X,Y]} Z \, ,
\end{align}
for $X,Y,Z \in \mathfrak{X}(M)$. 
Given a point $p \in M$ and $v,w \in T_pM$ such that they span a non-degenerate plane $\Pi$, the sectional curvature of $\Pi$ is defined by 
\begin{equation}
K(\Pi) := \frac{g(R(w,v)v,w)}{g(v,v) g(w,w) - g(v,w)^2} \, .
\end{equation}
It is easily seen that this definition is independent of the choice of generators of $\Pi$. 
Such a plane $\Pi \subset T_pM$ is called timelike provided $g_p |_{\Pi}$ has Lorentzian signature $(+,-)$.

\begin{definition}[Timelike sectional curvature bound]
\label{def: tl sec curv bd}
A spacetime $(M,g)$ is said to have a timelike sectional curvature bound from above (resp.\ below) by a constant $K \in \mathbb{R}$, denoted $\TSec \leq K$ (resp.\ $\geq$), if for all $p \in M$ and all timelike planes $\Pi \subseteq T_pM$, we have $K(\Pi) \leq K$ (resp.\ $\geq$).
\end{definition}

\begin{remark}[On the convention of signature]
In our signature convention, a bound from above (resp.\ below) on timelike sectional curvatures corresponds to a timelike curvature bound from above (resp.\ below) in the synthetic sense as pioneered in \cite{KS18}. 
Moreover, a space with curvature bounded, say, above by $K$ also has curvature bounded above by $K'$ for all $K' > K$, as one would intuitively expect. 
DeSitter space has constant curvature $K=-1$, which underlines the Riemannian expectation that negative curvature leads to the absence of conjugate points. 
See also \cite[Section 2.1]{GRZ26+} for more analytically motivated arguments for our signature convention and notations. 
\hfill$\blacksquare$ 
\end{remark}

Next, we recall a few basics about non-smooth Lorentzian geometry. We refer to \cite{KS18, BS23, BKR24} for more details. 

\begin{definition}[\LpLSn]
A Lorentzian pre-length space is a quintuple $(X,\mathsf d,\ll, \leq ,\ell)$ such that $(X,\mathsf d)$ is a metric space, $\leq$ is a reflexive and transitive relation, $\ll$ is a transitive relation contained in $\leq$, and the time separation function $\ell: X \times X \to [0,+\infty]$ satisfies $\ell(x,z) \geq \ell(x,y) + \ell(y,z)$ whenever $x \leq y \leq z$ and $\ell(x,y) > 0$ if and only if $x \ll y$.
\end{definition} 

The distance $\mathsf d$ is not important for us and serves mostly as a background tool. 
See \cite[Remark 2.8]{Octet} and \cite[Section 4.1]{BHNR25}  for how this definition relates to other notions of non-smooth spacetimes. 
Usually, in the non-smooth setting, the term geodesic refers to a globally maximising curve, while a geodesic on a spacetime, as a solution of the geodesic equation, is only a local maximiser of the time separation function. 
To minimise confusion, we try to emphasise the terminology `maximising geodesic'. 
In any case, for a maximising geodesic between $x$ and $y$, we use the notation $[x,y]$. 

Let us also recall the law of cosines, cf.\ \cite[Lemma 2.4]{BS23}, although only the case $K=0$ will be needed for our arguments.  
Let $p, q, r \in \mathbb{R}^{1,1}$ form an admissible causal triangle\footnote{By this we mean a triangle where at most one of the sides is allowed to be null. } (not necessarily in this order). 
Let $a = \max\{\ell (p, q), \ell (q, p)\} > 0$, $b = \max\{\ell (q, r), \ell (r, q)\} > 0$ and $c = \max \{\ell (p, r), \ell (r, p)\}$. 
Let $\omega=\ma_q(p,r)$ be the hyperbolic angle at $q$, and let $\sigma$ be its sign. 
Then we have: 
\begin{equation}
\label{eq: LOC}
a^2 + b^2 = c^2 - 2ab\sigma \cosh(\omega) \, .
\end{equation}

For the sake of  completeness, we restate the definition of triangle comparison here. 
For all other equivalent notions that we employ in this work, we refer the reader to \cite{BKR24}. 
The constant $D_K$ is given by $\frac{\pi}{\sqrt{K}}$ if $K > 0$, and $+ \infty$ if $K \leq 0$. 
By $\lm{K}$ we denote the Lorentzian model space of constant curvature $K$.

\begin{definition}[Triangle comparison]
\label{def: triangle comparison}
Let $X$ be a \LpLSn. An open subset $U$ is called a $(\leq K)$-comparison neighbourhood in the sense of triangle comparison if:
\begin{enumerate}[label=(\roman*)]
\item $\ell$ is continuous on $(U\times U) \cap \ell^{-1}([0,D_K))$, and this set is open; 
\item For all $x \ll y$ in $U$ with $\ell(x,y)<D_K$, there exists a connecting maximising geodesic inside $U$;
\item Let $\Delta (x,y,z)$ be a timelike triangle in $U$, with $p,q$ two points on the sides of $\Delta (x,y,z)$. Let $\Delta(\bar{x}, \bar{y}, \bar{z})$ be a comparison triangle in $\lm{K}$ for $\Delta (x,y,z)$ and $\bar{p},\bar{q}$ comparison points for $p$ and $q$, respectively. Then 
\begin{equation}
\label{eq: timelike triangle comparison inequality}
\ell(p,q) \geq \ell(\bar{p}, \bar{q}) \, .
\end{equation}
\end{enumerate}
$X$ is said to have timelike curvature bounded from above by $K$ if it is covered by $(\leq K)$-comparison neighbourhoods. 
$X$ is said to have timelike curvature bounded globally from above if $X$ itself is a comparison neighbourhood. 
\end{definition}

Note that within a $(\leq K)$-comparison neighbourhood, $\bar{p} \ll \bar{q}$ implies $p \ll q$. 
If $K=0$, we may refer to $X$ as having non-negative (resp.\ non-positive) curvature or as being non-negatively (resp.\ non-positively) curved. 
We may also call a triangle in a space $X$ flat if it satisfies equality in \eqref{eq: timelike triangle comparison inequality} with respect to its Euclidean comparison triangle for all comparison points. 
A globalisation theorem for curvature bounds from above has been recently established in \cite{EG25+}. 

Finally, recall that due to \cite{Harris} and the more recent \cite[Theorems 3.1 \& 3.2]{BKOR25}, Definition \ref{def: tl sec curv bd} and Definition \ref{def: triangle comparison} are equivalent for strongly causal spacetimes.

\section{The Ptolemy inequality}

In this section we give the definitions of the Ptolemy inequality and hyperbolic inversion on \LpLSsn. 
We also discuss some initial results.

\subsection{First properties and stability}

\begin{definition}[The Ptolemy inequality] \label{def: Ptolemy def}
A \LpLS $X$ is called Ptolemaic if for all points $x \leq y \leq z \leq w$ in $X$, the following inequality holds: 
\begin{equation} 
\label{eq: Ptolemy}
\ell(x,z) \, \ell(y,w) \geq \ell(x,y) \, \ell(z,w) + \ell(x,w) \, \ell(y,z) \, .
\end{equation}
\end{definition}


\begin{remark}[Basic properties of Ptolemaic spaces]
\label{rem: Ptolemy elem prop}
\mbox{}\par
\begin{itemize}
\item[(i)]
(Invariance by isometries) It is clear that the Ptolemy inequality and its equality cases are stable under isometries (where an isometry in the non-smooth setting is understood as a $\leq$- and $\ell$-preserving map). 

\item[(ii)] (Restrictions and rescaling) Any subset of a Ptolemaic \LpLSn, endowed with the relative topology and the restricted time separation function, is still a Ptolemaic \LpLSn. 
Moreover, a \LpLS $X$ is Ptolemaic if and only if the rescaled space $\lambda X$ is Ptolemaic for any $\lambda >0$.

\item[(iii)] (Restricting to timelike relations) Note that one can without loss of generality assume that $y \ll z$ holds for the quadruple considered in \eqref{eq: Ptolemy}. 
Indeed, otherwise, \eqref{eq: Ptolemy} reduces to $\ell(x,z)\ell(y,w) \geq \ell(x,y)\ell(z,w)$, which is trivially satisfied by the reverse triangle inequality. 

\item[(iv)] (Some equality cases) It is clear that four aligned points along a future-directed maximising causal geodesic in any \LpLS satisfy equality in the Ptolemy inequality. Moreover, equality is also trivially attained whenever two or more points do coincide.
\end{itemize}
\hfill$\blacksquare$ 
\end{remark}

\begin{remark}[Busemann concavity, stability and Finsler spacetimes]
Busemann concavity is another notion of (non-positive) curvature, well-suited for the setting of Finsler spacetimes. 
If the spacetime is genuinely Lorentzian, then Busemann concavity is equivalent to non-positive sectional curvature (since then the flag curvature agrees with the sectional one), cf.\ \cite{BEOR25+}. 

There are several notions of Lorentzian Gromov-Hausdorff (LGH for short) convergence in the literature, see for instance \cite{MinSuhr, MS}. Busemann concavity is not stable under any such convergence. 
Indeed, the family of hyperbolic $p$-norms 
\begin{equation*}
\|v\|_p:=\left(|v_0|^p - \sum_{i=1}^n|v_i|^p \right) ^{\frac1p}
\end{equation*}
on $\mathbb{R}^{1,n}$ for $p \in (1, + \infty)$ is easily seen to be Busemann concave. 
Yet, as $p \to 1$, they converge to a norm for which geodesics are no longer unique, in particular, Busemann concavity cannot hold. 
\hfill$\blacksquare$ 
\end{remark}

In contrast, the Ptolemaic condition is stable under LGH-convergence since one can pass to the limit in \eqref{eq: Ptolemy}. 
This fact suggests that Ptolemaic spaces shall single out Lorentzian manifolds out of Lorentz--Finsler ones as pointed out by \cite[Theorem 1.2]{BuFaWra} in the Riemannian setting.
Below, we will give a proof using the notion of convergence introduced in \cite{MS}. 



\begin{proposition}[Stability of Ptolemaic spaces]
\label{lem: ptolemy stability}
Let \((X_n, \ell_n)\) and \((X, \ell)\) be Lorentzian pre-length spaces and suppose that $\ell$ is continuous on $X$.
Suppose that \((X_n, \ell_n)\xrightarrow{\mathrm{LGH}} (X, \ell)\) and \((X_n,\ell_n)\) is Ptolemaic for all \(n\). 
Then \((X, \ell)\) is also Ptolemaic.
\end{proposition}

\begin{proof}
As in \cite[Theorem 10.4]{MS}, we only need to consider the case where all four points in question are vertices of causal diamonds of suitably chosen $\varepsilon$-nets in the definition of the convergence, cf.\ \cite[Definition 3.6]{MS}. 
For general points, one may approximate them by such vertices and then use continuity of $\ell$ to conclude. 

Suppose towards a contradiction that \(X\) is not Ptolemaic. 
Then there exist four points $x_1 \leq x_2 \ll x_3 \leq x_4$ in $X$ such that 
\begin{equation}
\label{eq: strict inequality contradiction}
\ell(x_1,x_3) \, \ell(x_2,x_4) < \ell(x_1,x_2) \, \ell(x_3,x_4) + \ell(x_1,x_4) \, \ell(x_2,x_3) \, .
\end{equation}

Since \eqref{eq: strict inequality contradiction} is a strict inequality, there exists $0 < \delta < \min\{\ell(x_i,x_j)>0,1\leq i<j \leq 4\}$ such that the inequality 
\begin{equation}
\label{yy}
 (\ell(x_1,x_3)+\delta) \, (\ell(x_2,x_4)+\delta) < (\ell(x_1,x_2)-\delta)^+ \, (\ell(x_3,x_4)-\delta)^+ + (\ell(x_1,x_4)-\delta) \, (\ell(x_2,x_3)-\delta) \, 
\end{equation}
holds, where $A^+:=\max\{A,0\}$ denotes the positive part of $A$. 

By the definition of LGH-convergence, for all sufficiently large $n$, there exist vertices $x_i^n, i=1,2,3,4$ in the respective $\varepsilon$-nets of $X_n$ such that 
\begin{equation}
\label{zz}
|\tilde{\ell}_n(x^n_i, x^n_j)-\tilde{\ell}(x_i, x_j)|<\delta,\quad 1\leq i<j \leq 4,
\end{equation}
where \[ \tilde{\ell}(x_i,x_j):=  
\begin{cases}
 \ell(x_i,x_j), &x_i\leq x_j,\\
 -\infty, & \text{otherwise}
\end{cases}
\]
and similarly for $\tilde \ell_n$. 
This immediately implies $x^n_1 \leq x^n_2 \leq x^n_3 \leq x^n_4$. 
Combining (\ref{yy}) and (\ref{zz}), we obtain
\begin{equation}
\label{xx}
\ell_n(x_1^n,x_3^n) \, \ell_n(x_2^n,x_4^n) < \ell_n(x_1^n,x_2^n) \, \ell_n(x_3^n,x_4^n) + \ell_n(x_1^n,x_4^n) \, \ell_n(x_2^n,x_3^n) \, ,
\end{equation}
which contradicts the fact that $X_n$ is Ptolemaic.  
\end{proof}

\subsection{The hyperbolic inversion}

In this section we introduce the hyperbolic inversion and discuss its properties in relation with curvature bounds on a spacetime. 
The hyperbolic inversion is a transformation closely related to the spherical inversion in metric geometry. 
A discussion of the hyperbolic inversion in Minkowski spacetime can be found in the book of Christodoulou \cite[Section 4.1]{Chr}. 
We give its definition in full generality. 
To this end, we first introduce the following terminology, cf.\ \cite[Definition 5.1]{GKS19}. 

\begin{definition}[Strong future timelike geodesic completeness]
A \LpLS $X$ is said to be (strongly) future timelike geodesically complete if every future-inextendible timelike geodesic (is maximising and) has infinite $\ell$-length.  
\end{definition} 

This is arguably a very strong condition, but it is essentially necessary for the hyperbolic inversion to be defined in a meaningful way. 
It is surely significantly more natural in a globally non-positively curved setting. 
Indeed, a space with non-positive curvature globally is already uniquely geodesic, cf.\ \cite[Theorem 4.7]{BNR25}, and any (local) geodesic is already maximising, cf.\ \cite[Remark 2.6]{EG25+}, so the only assumption apart from the curvature bound is that future-inextendible geodesics have infinite $\ell$-length. Examples of (globally non-positively curved) strongly future timelike geodesically complete spacetimes are given by Minkowski space and DeSitter space.

\begin{definition}[Hyperbolic inversion]
Let $X$ be a uniquely geodesic and strongly future timelike geodesically complete \LpLSn. 
Let $x \in X$ and $r > 0$. 
Let $p \in I^+(x)$ and denote by $\gamma_p$ the (unique and inextendible) maximising geodesic emanating from $x$ through $p$.
The map $H_{x,r} : I^+(x) \to I^+(x)$ is called the hyperbolic inversion through $x$ with radius $r$, defined by $H_{x,r}(p)=q$, where $q$ is the unique point on $\gamma_p$ such that $\ell(x,p)=\frac{r^2}{\ell(x,q)}$.
\end{definition}

Given any $r>0$ and $x \in X$, it is clear from the definition that $H_{x,r}$ is self-inverse, i.e., $H_{x,r}(H_{x,r}(y))=y$ for all $y \in I^+(x)$. 
In the next lemma, which has already been mentioned in \cite[Proposition 7, Section 4.1]{Chr}, we discuss the hyperbolic inversion in Minkowski space. 
Therein, we use the notation 
\[
\ma_x(y,z)=\arcosh \left (\frac{\langle y-x,z-x \rangle}{\|y-x\| \|z-x\|} \right )
\]
to denote the hyperbolic angle at $x$ between the points $y$ and $z$. 
We will not explicitly work with angles in a non-smooth setting, allowing us to forego that definition. 

\begin{lemma}[Hyperbolic inversion in Minkowski space]
\label{lem: hyperbolic inversion minkowski space}
Let $n \geq 1$ and let $x \ll z$ in $\mathbb R^{1,n}$ satisfy $\ell(x,z) = r$, then $H_{x,r}(I(x,z))=I^+(z)$ and $H_{x,r}(I^+(z))=I(x,z)$.
\end{lemma}

\begin{proof}
Without loss of generality, let $x$ be the origin in $\mathbb R^{1,n}$ and set $r=1$. Let $p \in I(x,z)$ and $q=H_{x,r}(p)$. 
Call $b=\ell(x,q), c= \ell(z,q), c'=\ell(p,z)$ and $\alpha=\measuredangle_x(z,q)$, then $\ell(x,p)=\frac1b$. Consider the triangle $\Delta(x,z,q)$ and notice that this triangle fits in the timelike plane inside $\mathbb R^{1,n}$ spanned by the timelike vectors $z-x$ and $p-x$. Plugging into the law of cosines, we have
\begin{equation*}
c^2 = 1 + b^2 - 2b\cosh(\alpha) = 1 + b(b-2\cosh(\alpha)) \, .
\end{equation*}
On the other hand, the law of cosines also yields 
\begin{equation*}
(c')^2 = 1 + \left(\frac1b\right)^2 - 2\frac1b\cosh(\alpha) = 1 + \frac1b \left( \frac1b-2\cosh(\alpha) \right) \, .
\end{equation*}
In particular, $c'= \frac{c}{b}$ and thus $c>0$ if and only if $c'>0$. 
\end{proof} 

The following proposition generalises the previous Lemma \ref{lem: hyperbolic inversion minkowski space} in the case of \LpLS with bounded curvature. 

\begin{proposition}[Hyperbolic inversion in curved spaces] 
Let $X$ be a future timelike geodesically complete \LpLS with non-positive curvature globally. 
Let $x \in X, r >0$ and let $z \in I^+(x)$ be such that $\ell(x,z)=r$. 
Then $H_{x,r}(I^+(z)) \subseteq I(x,z)$. 
\end{proposition}

\begin{proof}
Let $q \in I^+(z)$ and let $p=H_{x,r}(q)$. By construction we have  $p \in I^+(x)$. 
Consider the triangle $\Delta(x,z,q)$ in $X$ and its comparison triangle $\Delta(\bx,\bz,\bq)$ in the Minkowski plane. 
Since $p \in [x,q]$, there exists a comparison point $\bp \in [\bx,\bq]$ for $p$. 
Note that $H_{\bx, r}(\bq)=\bp$, see Figure \ref{fig: CAT inversion}. By assumption, $\ell(\bz,\bq)=\ell(z,q) > 0$ and hence by Lemma \ref{lem: hyperbolic inversion minkowski space} also $\ell(\bp,\bz) >0$. 
\begin{figure}
\begin{center}
\begin{tikzpicture}
\clip(-2.650284837275845,-0.3412336963647125) rectangle (3.05955569310921267,3.063627650137573);
\draw [samples=50,domain=-0.75:0.75,rotate around={90:(0,0)},xshift=0cm,yshift=0cm,line width=0.7pt] plot ({1*(1+(\x)^2)/(1-(\x)^2)},{1*2*(\x)/(1-(\x)^2)});
\draw [line width=0.7pt] (-0.38494639860674784,0.38494639860674784)-- (0.2644946687855832,1.0343874659990788);
\draw [line width=0.7pt] (0.2644946687855832,1.0343874659990788)-- (0.649441067392331,0.649441067392331);
\draw [line width=0.7pt] (0.649441067392331,0.649441067392331)-- (0,0);
\draw [line width=0.7pt] (0,0)-- (-0.38494639860674784,0.38494639860674784);
\draw [line width=0.7pt] (0.2644946687855832,1.0343874659990788)-- (-0.1761751732977016,1.7653766072101036);
\draw [line width=0.7pt] (0.2644946687855832,1.0343874659990788)-- (0,0);
\draw [dotted, line width=0.7pt,domain=0:3.5955569310921267] plot(\x,{(-0--0.649441067392331*\x)/0.649441067392331});
\draw [dotted, line width=0.7pt,domain=-3.050284837275845:0] plot(\x,{(-0--0.38494639860674784*\x)/-0.38494639860674784});
\draw [dotted, line width=0.7pt,domain=-0.28494639860674784:3.5955569310921267] plot(\x,{(--0.5--0.649441067392331*\x)/0.649441067392331});
\draw [dotted, line width=0.7pt,domain=-3.050284837275845:0.649441067392331] plot(\x,{(-0.5--0.38494639860674784*\x)/-0.3849463986067478});
\draw [dashed, line width=0.7pt,domain=-1.050284837275845:0] plot(\x,{(-0--1.7653766072101036*\x)/-0.1761751732977016});
\draw [line width=0.7pt] (-0.05709745903682406,0.5721493932883298)-- (0.2644946687855832,1.0343874659990788);
\begin{scriptsize}
\coordinate [circle, fill=black, inner sep=0.7pt, label=270: {$\bx$}] (x) at (0,0);
\coordinate [circle, fill=black, inner sep=0.7pt, label=90: {$\bz$}] (z) at (0.2644946687855832,1.0343874659990788);
\coordinate [circle, fill=black, inner sep=0.7pt, label=45: {$\bq$}] (q) at (-0.1761751732977016,1.7653766072101036);
\coordinate [circle, fill=black, inner sep=0.7pt, label=135: {$\bp$}] (p) at (-0.05709745903682406,0.5721493932883298);
\end{scriptsize}
\draw [line width=0.7pt] (x) -- (q);
\end{tikzpicture}
\end{center}
\caption{Hyperbolic inversion in the comparison configuration in the Minkowski plane.}
\label{fig: CAT inversion}
\end{figure}
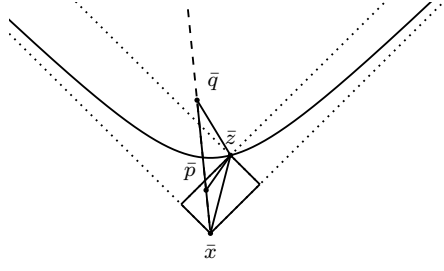
Then by triangle comparison, we have $0 <\ell(\bp,\bz) \leq \ell(p,z)$. 
This shows $p \in I(x,z)$ and hence $H_{x,r}(I^+(z)) \subseteq I(x,z)$. 
\end{proof}

In the Euclidean setting, one can pull back the Euclidean metric to induce a new distance on the space after inverting with respect to the unit sphere. 
This has been generalised to metric spaces in \cite{BHX08}. 
An analogous construction is possible in Lorentzian signature as well. 

\begin{definition}[Inversion time separation function]
Let $X$ be a uniquely geodesic and strongly future timelike geodesically complete \LpLSn. 
Let $p \in X$. 
The inversion time separation function (with respect to $p$) is given by $i_p: I^+(p) \times I^+(p) \to [0,+\infty]$, where
\begin{equation}
\label{eq: inversion ell}
    i_p(x,y) := \frac{\ell(y,x)}{\ell(p,x) \, \ell(p,y)} \, .
\end{equation} 
Moreover, we define a Lorentzian structure associated to this inversion by simply reversing the relations: define $x \leq_{i_p} y$ if and only if $y \leq x$ and similarly for the timelike relation. 
\end{definition}

An easy calculation shows the following identity in Minkowski space: 
\[
i_p(y,x)=\ell(H_{p,1}(x),H_{p,1}(y)), \qquad\text{see  Figure \ref{fig: inversion time sep}} \, 
. \]

\begin{figure}
\begin{center}
\begin{tikzpicture}
\clip(-2.5967881962699697,-0.3643638569956612) rectangle (2.444680343455278,3.590423482270273);
\draw [dotted, line width=0.7pt,domain=0:4.444680343455278] plot(\x,{(-0--1*\x)/1});
\draw [dotted, line width=0.7pt,domain=-3.5967881962699697:0] plot(\x,{(-0--1*\x)/-1});
\draw [dotted, samples=50,domain=-0.99:0.99,rotate around={90:(0,0)},xshift=0cm,yshift=0cm,line width=0.7pt] plot ({1*(1+(\x)^2)/(1-(\x)^2)},{1*2*(\x)/(1-(\x)^2)});
\draw [samples=50,domain=-0.99:0.99,rotate around={90:(0,0)},xshift=0cm,yshift=0cm,line width=0.7pt] plot ({1*(-1-(\x)^2)/(1-(\x)^2)},{1*(-2)*(\x)/(1-(\x)^2)});
\draw [dashed, line width=0.7pt,domain=-3.5967881962699697:0] plot(\x,{(-0--1.1877262874477794*\x)/-0.36865297960480825});
\draw [dashed, line width=0.7pt,domain=0:4.444680343455278] plot(\x,{(-0--3.087864414980768*\x)/0.8438661017767403});
\draw [line width=0.7pt] (-0.36865297960480825,1.1877262874477794)-- (0.8438661017767403,3.087864414980768);
\draw [line width=0.7pt] (0.09564610128564215,0.3499870322670203)-- (-0.28918751429546186,0.9317044259308644);
\begin{scriptsize}
\coordinate [circle, fill=black, inner sep=0.7pt, label=135: {$H_{p,1}(x)$}] (0) at  (-0.36865297960480825,1.1877262874477794);
\coordinate [circle, fill=black, inner sep=0.7pt, label=270: {$p$}] (0) at  (0,0);
\coordinate [circle, fill=black, inner sep=0.7pt, label=0: {$H_{p,1}(y)$}] (0) at  (0.8438661017767403,3.087864414980768);
\coordinate [circle, fill=black, inner sep=0.7pt, label=180: {$x$}] (0) at  (-0.28918751429546186,0.9317044259308644);
\coordinate [circle, fill=black, inner sep=0.7pt, label=45: {$y$}] (0) at (0.09564610128564215,0.3499870322670203);
\end{scriptsize}
\end{tikzpicture}
\end{center}
\caption{The inversion time separation in the Minkowski plane can be calculated via the law of cosines.}
\label{fig: inversion time sep}
\end{figure}
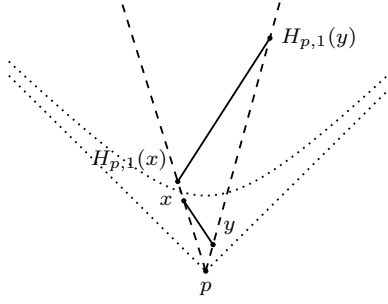

In general, $i_p$ may fail to satisfy the reverse triangle inequality. As in the positive signature, there is an immediate relationship between the original time separation being Ptolemaic and the inversion satisfying the reverse triangle inequality. 

\begin{proposition}[Hyperbolic inversion and being Ptolemaic]
Let $X$ be a uniquely geodesic and strongly future timelike geodesically complete \LpLSn. 
Then $(I^+(p),d, \ll_{i_p}, \leq_{i_p}, i_p)$ is a \LpLS if and only if $X$ is Ptolemaic. 
In particular, if $X$ is globally hyperbolic then so is $(I^+(p),d, \ll_{i_p}, \leq_{i_p}, i_p)$. 
\end{proposition}
\begin{proof}
Every property of a \LpLS except the reverse triangle inequality for $i_p$ is trivially satisfied. 
A simple calculation shows that $\ell$ being Ptolemaic is equivalent to $i_p$ satisfying the reverse triangle inequality. 
Indeed, simply plugging \eqref{eq: inversion ell} into $i_p(x,z) \geq i_p(x,y) + i_p(y,z)$ and rearranging yields the Ptolemy inequality. 

Clearly, $\leq$ is a partial order if and only if $\leq_{i_p}$ is, and the two spaces have the same causal diamonds (wherever it makes sense). 
\end{proof}

\subsection{The Ptolemy inequality in Minkowski space}

In the next lemma we prove that Minkowski space, denoted by $\mathbb R^{1,n}$ and endowed with its hyperbolic norm $\| \cdot \|$, is Ptolemaic and we discuss the rigidity of the Ptolemy inequality \eqref{eq: Ptolemy}. The strategy of the proof is similar to the one in the Euclidean space, where the Ptolemy inequality boils down to the triangle inequality after applying a circular inversion out of a point.

\begin{lemma}[Minkowski space is Ptolemaic]
\label{lem: Minkowski is Ptolemy}
Let $n \geq 1$ and $x \leq y \leq z \leq w$ be four points in $\mathbb R^{1,n}$. 
Then 
\begin{equation} 
\label{eq: Ptolemy in Minkowski}
\| z-x \| \, \| w-y \| \geq \| y-x \| \, \| w-z \| + \| w-x \| \, \|z-y\| \, .
\end{equation}
Moreover, equality holds in \eqref{eq: Ptolemy in Minkowski} if and only if at least two points coincide or $x,y,z,w$ lie on a (possibly degenerate) rectangular hyperbola or a timelike line inside a timelike \emph{2}-plane in $\mathbb R^{1,n}$.
\end{lemma}

\begin{proof}
We start with the proof of \eqref{eq: Ptolemy in Minkowski}. Without loss of generality we can assume that $x=0 \ll y \ll z \ll w$, indeed the general case is obtained by composing with a translation and a limiting procedure. 
We need to prove that
\[
\| z \| \, \| w- y \| \geq \| y \| \, \| w-z \| + \| w \| \,  \| z-y \| \, .
\]
Observe that since $y \ll w$, we can consider the hyperbolic inversion of $y$ centred at $x=0$ with radius $r= \|w\|$, obtaining 
\[
H_{x,r} (y) = \frac{\|w\|^2}{\|y\|^2} y \, ,
\]
and so by Lemma \ref{lem: hyperbolic inversion minkowski space} we have that $H_{x,r} (y) \geq w$, in particular, $\frac{\| w \|}{\|y\|} y \geq \frac{\|y\|}{\|w\|} w$. 
Now notice that
\[
\| w -y \| = \left \|   \frac{\| w \|}{\|y\|} y - \frac{\|y\|}{\|w\|} w \right \| \, .
\]
Hence the Ptolemy inequality follows as an application of the reverse triangle inequality:
\begin{align*}
\|z\| \, \| w-y \| & = \| z \|  \left \|   \frac{\| w \|}{\|y\|} y - \frac{\|y\|}{\|w\|} w \right \| \\
& = \left \| \frac{\|w \| \| z\|}{\|y \| } y - \frac{\|y\| \|w\|}{\|z\|} z + \frac{\|y\| \|w\|}{\|z\|} z - \frac{\|y\| \| z\|}{\|w\|} w \right \| \\
& \geq \left \| \frac{\|w \| \| z\|}{\|y \| } y - \frac{\|y\| \|w\|}{\|z\|} z \right \| + \left \| \frac{\|y\| \|w\|}{\|z\|} z - \frac{\|y\| \| z\|}{\|w\|} w \right \| \\
& =\| w\| \left \| \frac{\| z\|}{\|y \| } y - \frac{\|y\|}{\|z\|} z \right \| + \| y\|  \left \| \frac{\| w\|}{\|z \| } z - \frac{\|z\|}{\|w\|} w \right \| \\
& = \| w \| \,  \| z-y \| + \| y \| \, \| w-z \| \, .
\end{align*} 

We are left to discuss the equality case. 
If at least two points coincide, the Ptolemy inequality is an equality, hence we suppose that all the points are distinct. 
If the first three points (or the last three) are aligned along a null ray the equality is trivially verified (notice this case can be viewed as a degenerate rectangular hyperbola). Moreover if all relations are null, but no three points are aligned on a null ray, then one can easily verify (via the push-up property) that equality can never hold in \eqref{eq: Ptolemy in Minkowski}. In particular, we can assume that at least one relation is chronological.

We claim that the four points must lie on a timelike 2-plane inside $\mathbb R^{1,n}$. 
Since we have only four points we can restrict our attention to $\mathbb R^{1,2}$. 
To conclude that they lie on the same timelike 2-plane, it will be sufficient to check that the volume of the convex hull $\Delta$ spanned by the points $x,y,z$ and $w$ is zero. 
Using the Cayley--Menger determinant, cf.\ \cite[Section 9.7.3]{Berger}, we have
\[
\text{Vol}  (\Delta)^2 = - \frac{1}{288} \, \text{det} \begin{pmatrix}
0 & 1 & 1 & 1 & 1 \\
1 & 0 & \| y-x\|^2 & \|z-x\|^2 & \|w- x\|^2 \\
1 & \|y-x\|^2 & 0 & \|z-y\|^2 & \| w-y\|^2 \\
1 & \|z-x\|^2 & \|z-y\|^2 & 0 & \|w-z\|^2 \\
1 & \|w-x\|^2 & \| w-y\|^2 & \|w-z\|^2 & 0 \\ 
\end{pmatrix} \, .
\]
The determinant on the right-hand side, seen as a polynomial in the respective hyperbolic norms, can be factorized into a product of four terms, cf.\ \cite{Berger}, and one of them is precisely $\| z-x \| \, \| w-y \|- \| y-x \| \, \| w-z \| - \| w-x \| \, \|z-y\|$. 
Hence if \eqref{eq: Ptolemy in Minkowski} is an equality then $\text{Vol}(\Delta)=0$ as claimed. 

From now on we consider the four points in $\mathbb R^{1,1}$ and without loss of generality we take $x=0$. 
Notice that, along the lines of the proof, equality holds if and only if the points
\[
\frac{\|w\| \|z\|}{\|y\|} y \geq \frac{\|y\| \|w\|}{\|z\|}z \geq \frac{\|y\| \|z\|}{\|w\|} w\, ,
\]
are causally related and aligned, i.e.\ they lie on a geodesic segment. Then, necessarily there exists some $\lambda \in [0,1]$ such that
\begin{equation} \label{eq: Hyperbola}
    \lambda \, \frac{\|z\|^2}{\|y\|^2} \, ( \|w\| ^2 y - \|y\|^2 w) = \|w\|^2 z - \|z\|^2 w \, .
\end{equation}
If $x,y,w$ are aligned, then it is clear that $z$ must be aligned as well. 
Suppose that $x,y,w$ are not aligned.
Now notice that, if we keep $y$ and $w$ fixed, solving the equation \eqref{eq: Hyperbola} by $z= z_{\lambda}$ for $\lambda \in [0,1]$ gives a parametrization of a portion of a (possibly degenerate) rectangular hyperbola passing through $x=0$, $y$ and $w$.
\end{proof}

\begin{remark}[Edge cases of hyperbolas]
Arguing in a similar way to Lemma \ref{lem: Minkowski is Ptolemy}, one can show that any vector space with a Lorentzian scalar product is Ptolemaic. 

In the Minkowski plane with metric $d t^2 - d x^2$, a timelike rectangular hyperbola is obtained by selecting one of the two connected components of the set given by the equation $t^2-x^2 = -a$ for some $a \in (0, +\infty)$. Moreover, notice that a degenerate hyperbola or a timelike line can be achieved by LGH-convergence of rectangular hyperbolas as $a \rightarrow 0$ and $a \rightarrow +\infty$, respectively.
\hfill$\blacksquare$ 
\end{remark}

\section{Ptolemaic spacetimes}

Here we present some results about Ptolemaic spacetimes. 
In the first subsection, we aim to establish a link between the Ptolemy inequality and the non-positive timelike sectional curvature. 
We then go on to show that Ptolemaic spaces are future one-connected (using the uniqueness of geodesics) to globalise the curvature bound.

\subsection{Ptolemaic and non-positive curvature}

As in positive signature, it can be easily seen that spaces of non-positive curvature are Ptolemaic.

\begin{proposition}[Non-positive curvature implies Ptolemaic]
\label{prop: npc implies ptolemy}
Let $X$ be a globally non-positively curved \LpLSn. 
Then $X$ is Ptolemaic.
\end{proposition}

\begin{proof}
Without loss of generality, we only consider the case $x \leq y \ll z \leq w$ in $X$.
Consider comparison triangles\footnote{The triangles in question may not exactly fit into the framework of Definition \ref{def: triangle comparison}. However, $y \ll z$ forces both to be admissible causal triangles, and these are precisely the triangles used for so-called causal curvature bounds, cf.\ \cite[Theorem 4.2]{BKR24}. } for $\Delta(x,y,z)$ and $\Delta(x,z,w)$, respectively, arranged in such a way that they share the segment $[\bx,\bz]$ and are realised on opposite sides of this segment. 

We first assume that there exists a point of intersection between $[\bx,\bz]$ and $[\by,\bw]$, call it $\bp$.  
Let $p \in [x,z]$ such that $\bp$ is its comparison point.  By the non-positive curvature of $X$, we know that $\ell(y,p) \geq \ell(\by,\bp)$ and $\ell(p,w) \geq \ell(\bp,\bw)$. 
Then 
\[
\ell(\by,\bw) = \ell(\by,\bp) + \ell(\bp,\bw) \leq \ell(y,p) + \ell(p,w) \leq \ell(y,w) \, .
\]
By Lemma \ref{lem: Minkowski is Ptolemy}, we know that \eqref{eq: Ptolemy} is satisfied for $\bx \leq \by \leq \bz \leq \bw$. 
All involved other distances in the model spaces are equal to the corresponding distances in $X$. 

Suppose now that $[\bx,\bz] \cap [\by,\bw] = \emptyset$, i.e., $\bx \leq \by \ll \bz \leq \bw$ form a causal quadrilateral in $\lm{K}$ that is concave at $\bz$.  
Straightening out the quadrilateral in the spirit of Alexandrov's Lemma (cf. \cite[Lemma 3.4]{BR25+}) causes $\ell(\bx,\bz)$ and $\ell(\by,\bw)$ to decrease and keeps all other relevant distances the same. 
Hence we find ourselves in the first case (with the point of intersection being $\bz$) and the inequalities chain in the right direction. 
More precisely, keep $\bz$ fixed and denote by $\bx', \by'$ and $\bw'$ the moved points such that $\ell(\bx,\by)=\ell(\bx',\by'), \ell(\bx,\bw)=\ell(\bx',\bw')$ and $\ell(\by',\bw')=\ell(\by',\bz)+\ell(\bz,\bw')=\ell(\by,\bz) + \ell(\bz,\bw)$. 
Then $\ell(x,z)=\ell(\bx,\bz) \geq \ell(\bx',\bz)$ and $\ell(y,w) \geq \ell(y,z)+\ell(z,w) = \ell(\by',\bw')$ by Alexandrov's Lemma. 
Hence
\begin{equation}
\label{eq: npc implies ptolemy}
\ell(x,z) \, \ell(y,w) \geq \ell(\bx',\bz) \, \ell(\by',\bw') \geq \ell(\bx',\by') \, \ell(\bz,\bw') + \ell(\bx',\bw') \, \ell(\by',\bz) \, ,
\end{equation}
where the second inequality is due to Lemma \ref{lem: Minkowski is Ptolemy}. 
All distances on the right-hand side of \eqref{eq: npc implies ptolemy} are equal to the corresponding distances in $X$, and the claim follows. 
\end{proof}

Next, we are going to prove an analogous statement to \cite[Lemma 4.3]{BuFaWra}. In fact, we will prove a stronger statement more in the spirit of \cite[Lemma II.1.13]{BH99}.

\begin{proposition}[Strict hinge comparison in the model spaces]
\label{prop: Strict hinge comparison in the model spaces}
Let $K < K'$ and let $(x,p,y)$ be a timelike hinge\footnote{By a hinge $(x,p,y)$ we mean two maximising geodesics $[p,x], [p,y]$ emanating from $p$ and their inscribed angle.} in $\lm{K}$ with $p \ll x \ll y$, $\ell(p,y) < D_{K'}$ and $\alpha =\ma_p(x,y) > 0$. 
Let $(\bx,\bp,\by)$ be a comparison hinge in $\lm{K'}$. 
Then $\ell(\bx,\by) > \ell(x,y)$.
\end{proposition}

\begin{proof}
To improve readability, in this proof we emphasise the space in which the timelike future of a point is considered. 
Note that by standard hinge comparison \cite[Definition 3.14 \& Theorem 5.1]{BKR24}, the non-strict version of this inequality is automatically true.
We are going to introduce polar coordinates $(r,\theta)$ centred at $p$ and $\bp$. 
By elementary calculations, we have that the metric on the (timelike) future of a point $p \in \lm{K}$, in unified notation, has the form 
\begin{equation*}
g_K=dr^2 - f(r,K)^2d\theta^2 \, ,
\end{equation*}
where $f$ has domain $\{ (r,K) \mid K \in \mathbb{R}, r \in (0,D_K) \}$ and is defined by
\begin{equation*}
f(r,K):=
\begin{cases}
\frac{1}{\sqrt{-K}}\sinh(\sqrt{-K}r), & \text{ if } K<0 \, , \\
r, & \text{ if } K=0 \, , \\
\frac{1}{\sqrt{K}}\sin(\sqrt{K}r), & \text{ if } K>0 \, .
\end{cases}
\end{equation*} 
Using Taylor expansion, one can show that both $\frac{1}{\sqrt{-K}}\sinh(\sqrt{-K}r)$ and $\frac{1}{\sqrt{K}}\sin(\sqrt{K}r)$ converge to $r$ as $K \to 0$, i.e. $f$ is continuous in $K$. 
Moreover, $f$ is strictly monotonically decreasing in $K$. 

We define the map 
\begin{equation*}
\varphi: \{q \in I^+_{\lm{K}}(p) : \ell(p,q) < D_{K'}\} \cup \{p\}  \to I^+_{\lm{K'}}(\bp) \cup \{\bp\}
\end{equation*}
by $\varphi(p)=\bp$ and by sending a point $q$ with coordinates $(r,\theta)$ in $I^+_{\lm{K}}(p)$ with $\ell(p,q) < D_{K}$ to the point with the same coordinates in $I^+_{\lm{K'}}(\bp)$. 
Notice that, by construction, $\Delta(\bp,\bx,\by)$ is contained in the image of $\varphi$. 
As $\varphi$ is a radial isometry, we obtain $\varphi(x)=\bx$ and $\varphi(y)=\by$. 

Let $z= (r, \theta) \in I^+_{\lm{K}}(p)$ and let $v \in T_z\lm{K}$. We write
\[
v=v_1 \, \partial_r |_{z} + v_2 \, \partial_{\theta} |_z \, ,
\]
and so $g_K(v,v) = v_1^2 - f(r,K)^2v_2^2$. 
We also compute $g_{K'}(d_z \varphi [v], d_z \varphi [v]) = v_1^2 - f(r,K')^2v_2^2$, where $d_z \varphi$ is the differential of $\varphi$ at $z$. 
Clearly, by the monotonicity of $f$, we have 
\begin{equation}
\label{eq: norm increasment of varphi}
g_K(v,v) \geq g_{K'}(d_z \varphi [v], d_z \varphi [v])
\end{equation}
with equality if and only if $v_2 =0$, i.e., if and only if $v$ is proportional to $\partial_r$.

Let $\gamma :[0,1] \rightarrow \lm{K}$ be the (timelike) geodesic segment connecting $x$ and $y$. Then $\varphi \circ \gamma : [0,1] \rightarrow \lm{K'}$ is a smooth curve connecting $\bx$ and $\by$, and it is timelike because of \eqref{eq: norm increasment of varphi}. In particular, the curve $\varphi \circ \gamma$ lies inside $I^+_{\lm{K'}}(\bp)$. 
Thus, we get that 
\begin{equation} \label{eq: strict chain}
\ell(x,y) = \int_0^1 \sqrt{g_K(\dot \gamma_t, \dot \gamma_t ) } \d t \leq \int_0^1 \sqrt{g_{K'}( d_{\gamma_t}\varphi [\dot \gamma_t], d_{\gamma_t}\varphi[\dot \gamma_t] ) } \d t \leq \ell(\bx,\by) \, .
\end{equation}
Indeed, the last inequality follows by the definition of the time separation function as a supremum.
The first inequality in \eqref{eq: strict chain} is an equality, as we already mentioned, if and only if $\dot \gamma_t$ is proportional to $\partial_r$ but this is clearly the case if and only if the triangle $\Delta(p,x,y)$ would be degenerate, that is to say $\alpha=0$. 
\end{proof}

The next corollary is similar to \cite[Corollary 4.4]{BuFaWra}. 

\begin{corollary}[Strict hinge comparison in positively curved spacetimes]
\label{cor: Strict hinge comparison in positively curved spacetime}
Let $(M,g)$ be a strongly causal spacetime with $\TSec > 0$. 
Let $(x,p,y)$ be a timelike hinge in $M$ with $x \ll y$ and $\alpha=\ma_p(x,y) > 0$. 
Let $(\bx,\bp,\by)$ be a comparison hinge in the Minkowski plane. 
Then $\ell(x,y) > \ell(\bx,\by)$. 
\end{corollary}

\begin{proof}
Since the hinge is (initially) contained in a compact region $U$ of $M$, we may assume that $\TSec(M)|_U > c > 0$. 
Let $(\bx',\bp',\by')$ be a comparison hinge in $\lm{c}$ and let $(\bx,\bp,\by)$ be a comparison hinge in the Minkowski plane. 
Then by hinge comparison and Proposition \ref{prop: Strict hinge comparison in the model spaces}, we get $\ell(x,y) \geq \ell(\bx',\by') > \ell(\bx,\by)$. 
\end{proof}

In the following theorem we prove that a strongly causal Ptolemaic spacetime has non-positive timelike sectional curvature. Note that a key instance where the smoothness of $(M,g)$ is used is when we assume that the strict positivity of the curvature bound is continuous in Corollary \ref{cor: Strict hinge comparison in positively curved spacetime}.  

\begin{theorem}[Ptolemaic spacetimes are non-positively curved]
\label{thm: ptolemy implies NPC}
Let $(M,g)$ be a strongly causal Ptolemaic spacetime. 
Then $\TSec \leq 0$. 
\end{theorem}

\begin{proof}
Suppose towards a contradiction that there exists a point $p \in M$ and a timelike 2-plane $\Pi \subseteq T_pM$ such that $K(\Pi) > 0 $. 
Consider the restricted exponential map $\exp_p|_{\Pi}:\Pi \to M$, and
by choosing a sufficiently small neighbourhood $U$ of $0 \in \Pi$, we obtain a diffeomorphism to a two-dimensional (Lorentzian) submanifold $N=\exp_p(U) \hookrightarrow M$. We denote by $\| \cdot \|$ the hyperbolic norm obtained by $g$ on $\Pi$.

By the continuity of the curvature tensor, we can assume $U$ to be small enough such that $K(\Sigma) > \frac{S}2$ for all $y \in N$, where $\Sigma = \Sigma_y$ is the one and only plane in $T_yM$ that is tangent to $N$. 
Denote by $\ell_N$ the time separation function on $N$. Further denote the original time separation in $M$ by $\ell_M$, then clearly $\ell_N \leq \ell_M$. 
Therein, consider a timelike rectangular hyperbola symmetric to the $x$-axis but not passing through $0$ in $\Pi$.
Furthermore, consider two distinct timelike lines through $0$. 
Label the four points of intersection of these two lines with the hyperbola $v_1, v_2, v_3, v_4$, ordered by their time coordinates. 
Then necessarily $[v_1,v_3]$ and $[v_2,v_4]$ form geodesic segments in the Minkowski plane $\Pi$ and the four vectors form four non-zero angles at 0, see Figure \ref{fig: two lines through hyperbola}. 

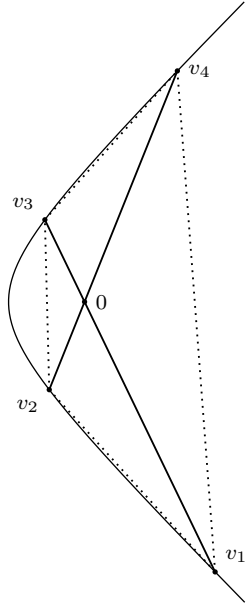
\begin{figure}
\begin{center}
\begin{tikzpicture}

\draw [samples=50,domain=-0.78:0.78,rotate around={0:(-2,0)},xshift=-2cm,yshift=0cm,line width=0.5pt] plot ({1*(1+(\x)^2)/(1-(\x)^2)},{1*2*(\x)/(1-(\x)^2)});

\draw [dotted, line width=0.7pt] (-0.4646137102691221,-1.165079850780002)-- (-0.5220617482171948,1.0882561628967764);
\draw [line width=0.7pt] (-0.5220617482171948,1.0882561628967764)-- (0,0);
\draw [line width=0.7pt] (0,0)-- (-0.4646137102691221,-1.165079850780002);
\draw [dotted, line width=0.7pt] (-0.4646137102691221,-1.165079850780002)-- (1.7177735992113656,-3.5807597706343035);
\draw [dotted, line width=0.7pt] (1.7177735992113656,-3.5807597706343035)-- (1.2210126948650655,3.0618495685585065);
\draw [line width=0.7pt] (1.2210126948650655,3.0618495685585065)-- (0,0);
\draw [line width=0.7pt] (0,0)-- (1.7177735992113656,-3.5807597706343035);
\draw [dotted, line width=0.7pt] (-0.5220617482171948,1.0882561628967764)-- (1.2210126948650655,3.0618495685585065);

\begin{scriptsize}
\coordinate [circle, fill=black, inner sep=0.7pt, label=0: {$0$}] (0) at (0,0);
\coordinate [circle, fill=black, inner sep=0.7pt, label=225: {$v_2$}] (ba) at (-0.4646137102691221,-1.165079850780002);
\coordinate [circle, fill=black, inner sep=0.7pt, label=0: {$v_4$}] (ba) at (1.2210126948650655,3.0618495685585065);
\coordinate [circle, fill=black, inner sep=0.7pt, label=135: {$v_3$}] (ba) at (-0.5220617482171948,1.0882561628967764);
\coordinate [circle, fill=black, inner sep=0.7pt, label=45: {$v_1$}] (ba) at (1.7177735992113656,-3.5807597706343035);
\end{scriptsize}

\end{tikzpicture}
\end{center}
\caption{Two lines pass through the origin and intersect the hyperbola in four points. 
}
\label{fig: two lines through hyperbola}
\end{figure} 

By construction, any hinge involving the origin and two of the $v_i$'s serves as a comparison hinge for the corresponding hinge at $p$. 
For $i=1,2,3,4$, set $x_i=\exp_p(v_i) \in N$.
Since $\exp_p$ is a radial isometry, we have $\ell_M(x_1,x_3)=\ell_N(x_1,x_3)=\|v_3 - v_1\|$ and $\ell_Mx_2,x_4)=\ell_N(x_2,x_4)=\|v_4 - v_2 \|$. 
By Corollary \ref{cor: Strict hinge comparison in positively curved spacetime}, we have $\ell_N(x_1,x_2) > \|v_2 - v_1\|$ as well as $\ell_N(x_3,x_4) > \|v_4 - v_3\|$. 
For the other distances involved in the Ptolemy inequality, the inequality is non-strict by standard hinge comparison, i.e.\ $\ell_N(x_2,x_3) \geq \|v_3- v_2\|$ and $\ell_N(x_1,x_4) \geq \|v_4- v_1\|$. 
By Lemma \ref{lem: Minkowski is Ptolemy}, the $v_i$'s satisfy equality in the Ptolemy inequality. 
Since $\ell_N \leq \ell_M$, the two strict inequalities obtained by Corollary \ref{cor: Strict hinge comparison in positively curved spacetime} give a contradiction to the Ptolemy inequality in $M$. 
\end{proof}

\subsection{Geodesic uniqueness and future one-connectedness}
This section aims to improve on the result of the previous section. 
Since Ptolemaic spacetimes satisfy TSec $\leq 0$, the timelike conjugate locus of every point is empty \cite[Proposition 11.13]{BeemBook}. 

In the next lemma we prove that even the whole timelike cut locus is empty. We refer to \cite[Chapter 9]{BeemBook} for the definition and basic properties of the timelike cut locus in Lorentzian geometry.
Indeed, the Ptolemaic condition permits us to globalise the fact that timelike geodesics are locally maximising. 
As a consequence of this fact we deduce that a Ptolemaic spacetime is uniquely geodesic.

\begin{lemma}[Globally maximising geodesics and empty cut locus]
\label{lem: Ptolemy imply maximising}
Let $(M,g)$ be a Ptolemaic spacetime, then all timelike geodesics are maximising geodesics. 
In particular, the timelike cut locus of any point is empty and there exists at most one future-directed maximising timelike geodesic connecting any two points.
\end{lemma}

\begin{proof}
We concentrate on future-directed geodesics since for past-directed geodesics the argument is the same. 
Let $\gamma :[0,b) \rightarrow M$ be a future-directed timelike geodesic parametrized by proper time and suppose that $0 \leq a < b$ is the parameter value of the first point along $\gamma$ in the (future) timelike cut locus of $\gamma(0) =p$. 
It is clear that $0<a$ and $\ell(\gamma_0, \gamma_t) = t$ for any $0 \leq t \leq a$, but 
$\ell(\gamma_0, \gamma_c) <c$ for all $a<c < b$. Let $\varepsilon>0$ be sufficiently small such that $a+ \varepsilon <b$ and $\gamma : [a- \varepsilon, a + \varepsilon ] \rightarrow M$ is a maximising timelike geodesic parametrized by proper time between its endpoints. 
Applying the Ptolemaic inequality to $\gamma(0) \ll \gamma(a-\varepsilon) \ll \gamma(a) \ll \gamma(a+ \varepsilon)$, we get
\[
2a \varepsilon \geq (a-\varepsilon) \varepsilon + \varepsilon \, \ell( \gamma_0, \gamma _{a+ \varepsilon}) \, ,
\]
and so we get that $a+ \varepsilon \geq \ell(\gamma_0, \gamma_{a+ \varepsilon})$. A straightforward application of the reverse triangle inequality tells that actually $a+ \varepsilon = \ell(\gamma_0, \gamma_{a+ \varepsilon})$, which contradicts the fact that $\gamma_a$ is in the timelike cut locus of $\gamma_0$. 

Suppose by contradiction that there exist two future-directed maximising timelike geodesics $\sigma, \gamma$ connecting $x$ to $y \in I^+(x)$. 
Up to reparametrization we can suppose that $\sigma, \gamma : [0,1] \rightarrow M$. 
By the first part of the proof, the geodesic $\gamma$ can be extended to the interval $[0, 1 + \varepsilon]$ in such a way that it remains a future-directed maximising timelike geodesic for some $\varepsilon >0$ sufficiently small. 
Consider
\[
\eta _t : = \begin{cases}
\sigma_t, & t \in [0,1] \, ,\\
\gamma_t, & t \in [1, 1+ \varepsilon] \, .
\end{cases}
\]
It is clear that $\eta$ is a future-directed maximising timelike geodesic between $\gamma_0$ and $\gamma _{1+ \varepsilon}$ since $\ell(\gamma_0, \gamma_{1+\varepsilon})$ coincides with the Lorentzian length of $\eta$. 
In particular, $\eta$ is a smooth curve and so $\dot \gamma _1= \dot \sigma _1$. Hence $\gamma= \sigma$ follows by a continuity argument.
\end{proof}

\begin{theorem}[Geodesic uniqueness and future one-connectedness]
\label{thm: future 1-connecedness}
Let $(M,g)$ be a globally hyperbolic Ptolemaic spacetime. Let $x \ll y$ in $M$, then there is a unique future-directed maximising timelike geodesic connecting $x$ to $y \in I^+(x)$. 
Moreover, $\exp_x : \Int (F_xM) \rightarrow I^+(x)$ is a diffeomorphism and $(M,g)$ is future one-connected. 
\end{theorem}

\begin{proof}
Let $x \ll y$ be two points in $M$. Since $(M,g)$ is a globally hyperbolic spacetime  there exists a future-directed maximising timelike geodesic $\gamma : [0,1] \rightarrow M$ connecting $x$ to $y$. 
By Lemma \ref{lem: Ptolemy imply maximising}, we get that $\gamma$ is the unique one. 

Let $x \in M$ be a point and consider $\exp_x : \Int (F_xM) \rightarrow I^+(x)$. 
It is clear that the map is surjective thanks to global hyperbolicity and it is injective thanks to the geodesic uniqueness we have just established. 
Moreover, the map is smooth since the timelike cut locus of $x$ is empty. Hence $\exp_x : \Int (F_xM) \rightarrow I^+(x)$ is a diffeomorphism. 

Using the property of unique geodesics, future one-connectedness is now easy to show. 
Consider two smooth future-directed timelike curves $\alpha , \beta : [0,1] \rightarrow M$ such that $\alpha _0= \beta _0 = x$ and $\alpha _1 = \beta _1 = y$. 
Then construct a timelike homotopy $\Gamma$ between $\alpha$ and $[x,y]$ such that $\Gamma(s,\cdot)$ is given by the concatenation of $[x,\alpha_s]$ and $\alpha|_{[s,1]}$. 
All intermediate curves are timelike by construction. 
A similar construction can be done between $[x,y]$ and $\beta$, and the composition of timelike homotopies is again a timelike homotopy. 
\end{proof}

We are now in a position to formulate our main result. 

\begin{theorem}[Ptolemy and global curvature]
    A globally hyperbolic spacetime is Ptolemaic if and only if it is globally non-positively curved.
\end{theorem}

\begin{proof}
The fact that global non-positive curvature implies that $M$ is Ptolemaic follows from Proposition \ref{prop: npc implies ptolemy}. 
Conversely, Theorem \ref{thm: ptolemy implies NPC} and Theorem \ref{thm: future 1-connecedness} imply that $M$ is future one-connected and satisfies $\TSec \leq 0$. Then we apply the globalisation result in \cite[Corollary 4.9]{EG25+} to get the claim. 
\end{proof}

\subsection{Rigidity of the hyperbolic inversion}

In this section we investigate some rigidity properties characterised by the hyperbolic inversion. 
To this end we make use of a four-point condition for curvature bounds from above, which has been developed more recently in \cite{BR25+}. 
For more details and in particular the equivalence of the following definition with the classical triangle comparison, we refer to this source.

\begin{definition}[Four-point condition]
\label{def: CBA four point}
Let $X$ be a \LpLSn. 
An open set $U$ is called a ($\leq K$)-comparison neighbourhood in the sense of the four-point condition if:
\begin{enumerate}[label=(\roman*)]
\item $\ell$ is continuous on $(U\times U)\cap \ell^{-1}([0,D_K))$, and this set is open;
\item For any $x_1\ll x_4$ in $U$, $x_2,x_3\in I(x_1,x_4):=I^+(x_1)\cap I^-(x_4)$, construct comparison triangles $\Delta(\hx_1, \hx_2, \hx_4)$ and $\Delta(\hx_1, \hx_3, \hx_4)$ realised on opposite sides of the line through $[\hx_1,\hx_4]$. 
Then $\ell(x_2,x_3) \geq \ell(\hx_2,\hx_3)$; 
\item For any $x_1 \ll x_2 \ll x_3 \ll x_4$ in $U$, construct comparison triangles $\Delta(\hx_1, \hx_2, \hx_3)$ and $\Delta(\hx_2, \hx_3, \hx_4)$ realised on the same side of the line through $[\hx_2,\hx_3]$. 
Then $\ell(x_1,x_4) \leq \ell(\hx_1,\hx_4)$. 
\end{enumerate}
\end{definition} 

The next proposition can be regarded as an additional point in the list of equivalent properties in \cite[Proposition 4.1]{BBCGRR26+}. 

\begin{proposition}[Rigidity of four-point configurations]
\label{prop: flat triangles}
Let $x \ll y \ll z \ll w$ be a quadruple of points in a ($\leq K$)-comparison neighbourhood $U$ of a Lorentzian pre-length space $X$. 
Assume the comparison configuration in \emph{Definition \ref{def: CBA four point} (iii)} satisfies equality in the defining equation. 
Then all four timelike triangles associated to the vertices $x \ll y \ll z \ll w$ satisfy equality in the defining equation for triangle comparison \eqref{eq: timelike triangle comparison inequality} for any choice of points $p$ and $q$.
\end{proposition}

\begin{proof}
Let $\hx, \hy ,\hz$ and $\hw$ be the vertices of the four-point comparison configuration as described in Definition \ref{def: CBA four point} (iii). 
As a consequence of the construction, the diagonals always intersect, say $[\hx, \hz] \cap [\hy, \hw] = \{ \hhm \}$. 
Moreover, we always have that $\Delta(\hy,\hz,\hw)$ is a comparison triangle for $\Delta(y,z,w)$. 

Since by assumption we have $\ell(x,w)=\ell(\hx,\hw)$, it follows that also $\Delta(\hx, \hy, \hz)$ is a comparison triangle for $\Delta(x,y,w)$. 
Denote by $m \in [x,z]$ the point for which $\hhm$ is a comparison point. 
Then by triangle comparison applied to $\Delta(x,y,z)$ and $\Delta(x,z,w)$, we infer $\ell(y,m) \geq \ell(\hy,\hhm)$ as well as $\ell(m,w) \geq \ell(\hhm,\hw)$. 
Combining this with the reverse triangle inequality, we get 
\begin{equation}
\label{eq: rigidity four-point ptolemy}
\ell(y,w) = \ell(\hy,\hw) = \ell(\hy,\hhm) + \ell(\hhm,\hw) \leq \ell(y,m) + \ell(m,w) \leq \ell(y,w) \, , 
\end{equation}
which forces equality in the triangle comparison for $m$ in $\Delta(x,y,z)$ and $\Delta(x,z,w)$. 
Now by \cite[Proposition 4.1]{BBCGRR26+}, this already implies that these two triangles satisfy full equality in \eqref{eq: timelike triangle comparison inequality}. 
Indeed, using the first variation formula \cite[Theorem 3.6]{BBCGRR26+} one can construct an $\ell$-isometry from the `filled in' triangle $\Delta(\hx,\hy,\hz)$ back into $U$ that maps the sides onto $\Delta(x,y,z)$.
We can proceed similarly with the point $m' \in [y,w]$ for which $\hhm$ is a comparison point as well, giving the claim for all four triangles in question.
\end{proof}

For the next proposition, we briefly recall the following terminology, cf.\ \cite[Definition 3.6]{BS23}.
On a sufficiently regular Lorentzian pre-length space (for instance a global curvature bound from above is enough) one can construct the metric space $(\mathcal{D}_p^+, \ma_x)$ given by the future-directed timelike directions at $x \in X$ (modulo the equivalence relation of two geodesics having zero angle between them) endowed with the angle metric.

\begin{definition}[{Minkowski cone, \cite[Definition 3.16]{BS23}}]
Given a metric space $(X, \mathsf d)$, the Minkowski cone over $X$ is the Lorentzian pre-length space defined as the usual metric cone $\Cone(X)=([0,+ \infty) \times X)/( \{0\} \times X)$ as metric space (cf. \emph{\cite[Definition 3.6.16]{BurBurIva}}) endowed with
\begin{itemize}
    \item the time separation $\ell_{C} ((s,x),(t,y))=\sqrt{s^2+t^2-2st\cosh(\mathsf d(x,y))}$ ,
\item the causal relation $(s,x) \leq_C (t,y)$ if and only if $ s^2+t^2-2st\cosh(\mathsf d(x,y)) \geq 0 \; \& \; s \leq t$ ,
\item the chronological relation $(s,x) \ll_C (t,y)$ if and only if $\ell_{C}((s,x),(t,y)) > 0$ .
\end{itemize}
\end{definition}

\begin{proposition}[Timelike future is isometric to Minkowski cone] 
\label{prop: isometry into the minkowski cone}
Let $(X, \ell)$ be a Lorentzian pre-length space with non-positive curvature globally and such that $(I^+(p),i_p)$ is a geodesic \LpLSn.
Then $I^+(p)$ is $\ell$-isometric to $\Cone(\mathcal D_p^+)$, i.e.\ the Minkowski cone over the space of directions at $p$. 
\end{proposition}

\begin{proof}
Let $x \ll z$ in $I^+(p)$ and $y$ be a point on the $i_p$-geodesic from $z$ to $x$ (recall that the causal order of $i_p$ is time-dual to the one on $X$), i.e., $i_p(z,x)=i_p(z,y)+i_p(y,x)$. 
Unpacking the definition, this leads to 
\[
\ell(p,y) \, \ell(x,z) = \ell(p,z) \, \ell(x,y) + \ell(p,x) \, \ell(y,z) \, ,
\]
i.e., the quadruple $p \leq x \leq y \leq z$ satisfies equality in the Ptolemy inequality. 
We now claim that this quadruple also satisfies equality in the four-point condition. 
To this end, let $\hp \leq \hx \leq \hy \leq \hz$ be a comparison configuration in the Minkowski plane in the spirit of Definition \ref{def: CBA four point} (iii). 
If we assume towards a contradiction that $\ell(p,z) < \ell(\hp,\hz)$, then
\begin{align*}
    \ell(\hp,\hy) \, \ell(\hx,\hz) & = \ell(p,y) \, \ell(x,z) \\ & = \ell(p,z) \, \ell(x,y) + \ell(p,x) \, \ell(y,z) \\ &< \ell(\hp,\hz)\, \ell(\hx,\hy) + \ell(\hp,\hx) \, \ell(\hy,\hz) \, ,
\end{align*}
violating the Ptolemy inequality in the Minkowski plane. 
Thus, we obtain that $\Delta(p,x,z)$ is flat by Proposition \ref{prop: flat triangles}. 

We claim the following map is a $\ell$-isometry. 
Let $f: I^+(p) \to \Cone(\mathcal{D}_p^+)$ be defined by $x \mapsto (\ell(p,x),v_x)$, where $v_x$ is the equivalence class in $\mathcal{D}_p^+$ that contains $[p,x]$.
This is well-defined as any pair of geodesics that enclose an angle of zero have the same angle to any other geodesic by the triangle inequality for angles, cf.\ \cite[Theorem 3.1]{BS23}. 
Let $x,y \in I^+(p)$ and consider the triangle $\Delta(p,x,y)$, which we know to be flat. 
In particular, this implies that the angle at $p$ is equal to its comparison angle in the Minkowski plane. 
Denote this angle by $\omega$.
The law of cosines for $\omega$ in the comparison triangle $\Delta(\bp,\bx,\by)$ reads: 
\begin{equation}
\ell(x,y)^2 = \ell(p,x)^2 + \ell(p,y)^2 - 2\ell(p,x)\ell(p,y)\cosh(\omega) \, .
\end{equation}
By definition, the time separation in $\Cone(\mathcal{D}_p^+)$ is given by 
\begin{equation}
\ell_C((s,v),(t,w))=\sqrt{s^2+t^2-2st\cosh(\mathsf d(v,w))} \, .
\end{equation}
Reading this formula with $f(x)$ and $f(y)$, we clearly have $s=\ell(p,x), t=\ell(p,y)$ and $\mathsf d(v_x,v_y)=\omega$, hence $\ell(x,y)=\ell_C(f(x),f(y))$. 
\end{proof}

\begin{theorem}[Geodesic inversion and vanishing sectional curvature]
Let $(M,g)$ be a spacetime and $p \in M$ such that $(I^+(p),i_p)$ is a geodesic \LpLSn. Then $K(\Pi)=0$ for any timelike \emph{2}-plane $\Pi \subset T_pM$.
\end{theorem}

\begin{proof}
Fix a normal neighbourhood $U$ of $p$ in $M$. Let $\Pi \subset T_pM$ be a timelike 2-plane and let $v, w$ be two future-directed timelike vectors such that $v-w \in \Int(F_pM)$ and $\Pi= \text{span}(v,w)$. Working in $U$, one can Taylor expand the time separation function on $t>s$ to obtain
\begin{equation}
    \ell( \exp_p (tv), \exp_p(sw)) ^2  = \| tv- sw\|^2 _{F_g} - \frac{1}{6} g(R(v,w)v,w) s^2t^2 + o( (t^2+s^2)^2)
\end{equation}
as $s^2+t^2 \rightarrow 0$, see \cite[Page 7]{mccann2025tradinglinearityellipticitynonsmooth}. 
The fact that $(I^+(p),i_p)$ is a geodesic \LpLS in particular implies that $i_p$ satisfies the reverse triangle inequality, which in turn yields that $(M,g)$ is Ptolemaic. 
Thus, $\TSec \leq 0$ by Theorem \ref{thm: ptolemy implies NPC}. 
Clearly, a local version of Proposition \ref{prop: isometry into the minkowski cone} yields an $\ell$-isometry from $U \cap I^+(p)$ onto a subset of the Minkowski cone over the space of directions, which in the smooth setting is $\Int(F_pM)$, cf.\ \cite[Lemma 3.24]{BS23}. 
Thus, as a consequence of Proposition \ref{prop: isometry into the minkowski cone} we find that $\ell( \exp_p (tv), \exp_p(sw)) ^2  = \| tv- sw\|^2 _{F_g}$ when $t,s$ are sufficiently small. 
Hence it is clear that $K(\Pi)=0$.
\end{proof}

It turns out that vanishing timelike sectional curvature in fact implies that the sectional curvature vanishes entirely, cf.\ e.g.\ \cite{NomDaj}. 
Thus, using the standard classification theorem for constant curvature spaces, cf.\ \cite[Corollary 8.26]{One83}
we obtain the following corollary. 

\begin{corollary}[Characterising Minkowski space]
Let $(M,g)$ be a geodesically complete and simply connected spacetime such that $(I^+(p), i_p)$ is a geodesic \LpLS for all $p \in M$. 
Then $M$ is isometric to Minkowski space. 
\end{corollary}

\begin{ack}
We want to thank Tobias Beran and Darius Er\"os for interesting discussions in the initial stages of this project. 

The authors acknowledge the support of the European Union - NextGenerationEU, in the framework of the PRIN Project `Contemporary perspectives on geometry and gravity' (code 2022JJ8KER – CUP G53D23001810006). The views and opinions expressed are solely those of the authors and do not necessarily reflect those of the European Union, nor can the European Union be held responsible for them. 

The second author is also supported by the Program of China Scholarship Council grant (No.202406340143). 
\end{ack}


\bibliography{bibliography} 

@book {Chr,
    AUTHOR = {Christodoulou, Demetrios},
     TITLE = {Mathematical problems of general relativity. {I}},
    SERIES = {Zurich Lectures in Advanced Mathematics},
 PUBLISHER = {European Mathematical Society (EMS), Z\"urich},
      YEAR = {2008},
     PAGES = {x+147},
      ISBN = {978-3-03719-005-0},
   MRCLASS = {83C05 (35Q75 58J45 83-02)},
  MRNUMBER = {2391586},
MRREVIEWER = {Alan\ D.\ Rendall},
       DOI = {10.4171/005},
       URL = {https://doi.org/10.4171/005},
}

@book {Berger,
    AUTHOR = {Berger, Marcel},
     TITLE = {Geometry. {I}},
    SERIES = {Universitext},
      NOTE = {Translated from the French by M. Cole and S. Levy},
 PUBLISHER = {Springer-Verlag, Berlin},
      YEAR = {1987},
     PAGES = {xiv+428},
      ISBN = {3-540-11658-3},
   MRCLASS = {51-01},
  MRNUMBER = {882541},
       DOI = {10.1007/978-3-540-93815-6},
       URL = {https://doi.org/10.1007/978-3-540-93815-6},
}

@article {BuFaWra,
    AUTHOR = {Buckley, S. M. and Falk, K. and Wraith, D. J.},
     TITLE = {Ptolemaic spaces and {CAT}(0)},
   JOURNAL = {Glasg. Math. J.},
  FJOURNAL = {Glasgow Mathematical Journal},
    VOLUME = {51},
      YEAR = {2009},
    NUMBER = {2},
     PAGES = {301--314},
      ISSN = {0017-0895,1469-509X},
   MRCLASS = {53C23 (51F99 53C20)},
  MRNUMBER = {2500753},
MRREVIEWER = {Hanspeter\ Fischer},
       DOI = {10.1017/S0017089509004984},
       URL = {https://doi.org/10.1017/S0017089509004984},
}

@misc{MS,
      title={Lorentzian {G}romov--{H}ausdorff convergence and pre-compactness}, 
      author={Andrea Mondino and Clemens Sämann},
      year={2025},
      eprint={2504.10380},
      archivePrefix={arXiv},
      primaryClass={math.DG},
      url={https://arxiv.org/abs/2504.10380}, 
      note={Preprint: \url{https://arxiv.org/abs/2504.10380}}
}

@article {BKR24,
    AUTHOR = {Beran, Tobias and Kunzinger, Michael and Rott, Felix},
     TITLE = {On curvature bounds in {L}orentzian length spaces},
   JOURNAL = {J. Lond. Math. Soc. (2)},
  FJOURNAL = {Journal of the London Mathematical Society. Second Series},
    VOLUME = {110},
      YEAR = {2024},
    NUMBER = {2},
     PAGES = {Paper No. e12971, 41},
      ISSN = {0024-6107,1469-7750},
   MRCLASS = {53C23 (53B30 53C50)},
  MRNUMBER = {4781260},
MRREVIEWER = {Argam\ Ohanyan},
       DOI = {10.1112/jlms.12971},
       URL = {https://doi.org/10.1112/jlms.12971},
}

@article {BS23,
    AUTHOR = {Beran, Tobias and S\"amann, Clemens},
     TITLE = {Hyperbolic angles in {L}orentzian length spaces and timelike
              curvature bounds},
   JOURNAL = {J. Lond. Math. Soc. (2)},
  FJOURNAL = {Journal of the London Mathematical Society. Second Series},
    VOLUME = {107},
      YEAR = {2023},
    NUMBER = {5},
     PAGES = {1823--1880},
      ISSN = {0024-6107,1469-7750},
   MRCLASS = {53B30 (28A75 51K10 53C23 53C50 53C80)},
  MRNUMBER = {4585303},
MRREVIEWER = {Benjam\'in\ Olea},
       DOI = {10.1112/jlms.12726},
       URL = {https://doi.org/10.1112/jlms.12726},
}

@article {BHNR25,
    AUTHOR = {Beran, Tobias and Harvey, John and Napper, Lewis and Rott,
              Felix},
     TITLE = {A {T}oponogov globalisation result for {L}orentzian length
              spaces},
   JOURNAL = {Math. Ann.},
  FJOURNAL = {Mathematische Annalen},
    VOLUME = {392},
      YEAR = {2025},
    NUMBER = {3},
     PAGES = {3447--3478},
      ISSN = {0025-5831,1432-1807},
   MRCLASS = {53C50 (53B30 53C23 53C80)},
  MRNUMBER = {4939668},
       DOI = {10.1007/s00208-025-03167-w},
       URL = {https://doi.org/10.1007/s00208-025-03167-w},
}

@misc{Octet,
      title={A nonlinear d' {A}lembert comparison theorem and causal differential calculus on metric measure spacetimes}, 
      author={Tobias Beran and Mathias Braun and Matteo Calisti and Nicola Gigli and Robert J. McCann and Argam Ohanyan and Felix Rott and Clemens Sämann},
      year={2024},
      eprint={2408.15968},
      archivePrefix={arXiv},
      primaryClass={math.DG},
      url={https://arxiv.org/abs/2408.15968}, 
    note={Preprint: \url{https://arxiv.org/abs/2408.15968}}
}

@misc{BR25+,
      title={Reshetnyak {M}ajorisation and discrete upper curvature bounds for {L}orentzian length spaces}, 
      author={Tobias Beran and Felix Rott},
      year={2025},
      eprint={2509.05224},
      archivePrefix={arXiv},
      primaryClass={math.DG},
      url={https://arxiv.org/abs/2509.05224}, 
    note={Preprint: \url{https://arxiv.org/abs/2509.05224}}
}

@misc{EG25+,
      title={A synthetic {L}orentzian {C}artan-{H}adamard theorem}, 
      author={Darius Erös and Sebastian Gieger},
      year={2025},
      eprint={2506.22197},
      archivePrefix={arXiv},
      primaryClass={math.MG},
      url={https://arxiv.org/abs/2506.22197}, 
    note={Preprint: \url{https://arxiv.org/abs/2506.22197}}
}

@Book{BH99,
 Author = {Bridson, Martin R. and Haefliger, Andr{\'e}},
 Title = {Metric spaces of non-positive curvature},
 FSeries = {Grundlehren der Mathematischen Wissenschaften},
 Series = {Grundlehren Math. Wiss.},
 ISSN = {0072-7830},
 Volume = {319},
 ISBN = {3-540-64324-9},
 Year = {1999},
 Publisher = {Berlin: Springer},
 Language = {English},
 zbMATH = {1385418},
 Zbl = {0988.53001}
}

@article {FLS07,
    AUTHOR = {Foertsch, Thomas and Lytchak, Alexander and Schroeder, Viktor},
     TITLE = {Nonpositive curvature and the {P}tolemy inequality},
   JOURNAL = {Int. Math. Res. Not. IMRN},
  FJOURNAL = {International Mathematics Research Notices. IMRN},
      YEAR = {2007},
    NUMBER = {22},
     PAGES = {Art. ID rnm100, 15},
      ISSN = {1073-7928,1687-0247},
   MRCLASS = {53C23 (53C70)},
  MRNUMBER = {2376212},
MRREVIEWER = {Daniel\ P.\ Groves},
       DOI = {10.1093/imrn/rnm100},
       URL = {https://doi.org/10.1093/imrn/rnm100},
}

@Article{BKOR25,
 Author = {Beran, Tobias and Kunzinger, Michael and Ohanyan, Argam and Rott, Felix},
 Title = {The equivalence of smooth and synthetic notions of timelike sectional curvature bounds},
 FJournal = {Proceedings of the American Mathematical Society},
 Journal = {Proc. Am. Math. Soc.},
 ISSN = {0002-9939},
 Volume = {153},
 Number = {2},
 Pages = {783--797},
 Year = {2025},
 Language = {English},
 DOI = {10.1090/proc/17022},
 zbMATH = {7972338}
}

@article {KS18,
    AUTHOR = {Kunzinger, Michael and S\"amann, Clemens},
     TITLE = {Lorentzian length spaces},
   JOURNAL = {Ann. Global Anal. Geom.},
  FJOURNAL = {Annals of Global Analysis and Geometry},
    VOLUME = {54},
      YEAR = {2018},
    NUMBER = {3},
     PAGES = {399--447},
      ISSN = {0232-704X,1572-9060},
   MRCLASS = {53C23 (53B30 53C50 53C80)},
  MRNUMBER = {3867652},
MRREVIEWER = {Benjam\'in\ Olea},
       DOI = {10.1007/s10455-018-9633-1},
       URL = {https://doi.org/10.1007/s10455-018-9633-1},
}

@book {BurBurIva,
    AUTHOR = {Burago, Dmitri and Burago, Yuri and Ivanov, Sergei},
     TITLE = {A course in metric geometry},
    SERIES = {Graduate Studies in Mathematics},
    VOLUME = {33},
 PUBLISHER = {American Mathematical Society, Providence, RI},
      YEAR = {2001},
     PAGES = {xiv+415},
      ISBN = {0-8218-2129-6},
   MRCLASS = {53C23},
  MRNUMBER = {1835418},
MRREVIEWER = {Mario\ Bonk},
       DOI = {10.1090/gsm/033},
       URL = {https://doi.org/10.1090/gsm/033},
}

@article {NomDaj,
    AUTHOR = {Dajczer, Marcos and Nomizu, Katsumi},
     TITLE = {On sectional curvature of indefinite metrics. {II}},
   JOURNAL = {Math. Ann.},
  FJOURNAL = {Mathematische Annalen},
    VOLUME = {247},
      YEAR = {1980},
    NUMBER = {3},
     PAGES = {279--282},
      ISSN = {0025-5831,1432-1807},
   MRCLASS = {53C50},
  MRNUMBER = {568993},
MRREVIEWER = {V.\ G.\ Kopp},
       DOI = {10.1007/BF01348960},
       URL = {https://doi.org/10.1007/BF01348960},
}

@book {BeemBook,
    AUTHOR = {Beem, John K. and Ehrlich, Paul E. and Easley, Kevin L.},
     TITLE = {Global {L}orentzian geometry},
    SERIES = {Monographs and Textbooks in Pure and Applied Mathematics},
    VOLUME = {202},
   EDITION = {Second},
 PUBLISHER = {Marcel Dekker, Inc., New York},
      YEAR = {1996},
     PAGES = {xiv+635},
      ISBN = {0-8247-9324-2},
   MRCLASS = {53C50 (53-02 83-02)},
  MRNUMBER = {1384756},
MRREVIEWER = {Peter\ R.\ Law},
}

@article {Min19,
    AUTHOR = {Minguzzi, Ettore},
     TITLE = {Lorentzian causality theory},
   JOURNAL = {Living Rev. Relativ.},
  FJOURNAL = {Living Reviews in Relativity},
    VOLUME = {22},
      YEAR = {2019},
    NUMBER = {3},
     PAGES = {},
      ISSN = {},
    MRCLASS = {},
  MRNUMBER = {},
MRREVIEWER = {},
    DOI = {10.1007/s41114-019-0019-x},
       URL = {https://doi.org/10.1007/s41114-019-0019-x},
}

@article {MinSuhr,
    AUTHOR = {Minguzzi, E. and Suhr, S.},
     TITLE = {Lorentzian metric spaces and their {G}romov--{H}ausdorff
              convergence},
   JOURNAL = {Lett. Math. Phys.},
  FJOURNAL = {Letters in Mathematical Physics},
    VOLUME = {114},
      YEAR = {2024},
    NUMBER = {3},
     PAGES = {Paper No. 73, 63},
      ISSN = {0377-9017,1573-0530},
   MRCLASS = {53C50 (51F99 83C45)},
  MRNUMBER = {4752400},
MRREVIEWER = {S.\ M. B. Kashani},
       DOI = {10.1007/s11005-024-01813-z},
       URL = {https://doi.org/10.1007/s11005-024-01813-z},
}

@article {GKS19,
    AUTHOR = {Grant, James D. E. and Kunzinger, Michael and S\"amann,
              Clemens},
     TITLE = {Inextendibility of spacetimes and {L}orentzian length spaces},
   JOURNAL = {Ann. Global Anal. Geom.},
  FJOURNAL = {Annals of Global Analysis and Geometry},
    VOLUME = {55},
      YEAR = {2019},
    NUMBER = {1},
     PAGES = {133--147},
      ISSN = {0232-704X,1572-9060},
   MRCLASS = {53C23 (53B30 53C50 53C80 83C75)},
  MRNUMBER = {3916126},
MRREVIEWER = {Guanghan\ Li},
       DOI = {10.1007/s10455-018-9637-x},
       URL = {https://doi.org/10.1007/s10455-018-9637-x},
}

@article {BNR25,
    AUTHOR = {Beran, Tobias and Napper, Lewis and Rott, Felix},
     TITLE = {Alexandrov's patchwork and the {B}onnet-{M}yers theorem for
              {L}orentzian length spaces},
   JOURNAL = {Trans. Amer. Math. Soc.},
  FJOURNAL = {Transactions of the American Mathematical Society},
    VOLUME = {378},
      YEAR = {2025},
    NUMBER = {4},
     PAGES = {2713--2743},
      ISSN = {0002-9947,1088-6850},
   MRCLASS = {53C50 (51K10 53B30 53C23)},
  MRNUMBER = {4880460},
       DOI = {10.1090/tran/9372},
       URL = {https://doi.org/10.1090/tran/9372},
}

@article {BHX08,
    AUTHOR = {Buckley, Stephen M. and Herron, David A. and Xie, Xiangdong},
     TITLE = {Metric space inversions, quasihyperbolic distance, and uniform
              spaces},
   JOURNAL = {Indiana Univ. Math. J.},
  FJOURNAL = {Indiana University Mathematics Journal},
    VOLUME = {57},
      YEAR = {2008},
    NUMBER = {2},
     PAGES = {837--890},
      ISSN = {0022-2518,1943-5258},
   MRCLASS = {30F45 (54E35)},
  MRNUMBER = {2414336},
MRREVIEWER = {Petra\ Bonfert-Taylor},
}

@article {Sch40,
    AUTHOR = {Schoenberg, I. J.},
     TITLE = {On metric arcs of vanishing {M}enger curvature},
   JOURNAL = {Ann. of Math. (2)},
  FJOURNAL = {Annals of Mathematics. Second Series},
    VOLUME = {41},
      YEAR = {1940},
     PAGES = {715--726},
      ISSN = {0003-486X},
   MRCLASS = {27.2X},
  MRNUMBER = {2903},
MRREVIEWER = {L.\ M.\ Blumenthal},
       DOI = {10.2307/1968849},
       URL = {https://doi.org/10.2307/1968849},
}

@article {Sch52,
    AUTHOR = {Schoenberg, I. J.},
     TITLE = {A remark on {M}. {M}. {D}ay's characterization of
              inner-product spaces and a conjecture of {L}. {M}.
              {B}lumenthal},
   JOURNAL = {Proc. Amer. Math. Soc.},
  FJOURNAL = {Proceedings of the American Mathematical Society},
    VOLUME = {3},
      YEAR = {1952},
     PAGES = {961--964},
      ISSN = {0002-9939,1088-6826},
   MRCLASS = {46.3X},
  MRNUMBER = {52035},
MRREVIEWER = {Chr.\ Pauc},
       DOI = {10.2307/2031742},
       URL = {https://doi.org/10.2307/2031742},
}

@misc{BBCGRR26+,
      title={A Splitting Theorem for non-positively curved {L}orentzian spaces}, 
      author={Barton, Joe and Beran, Tobias and Che, Mauricio and Gieger, Sebastian and R\"ohrig, Jona and Rott, Felix},
      year={2026},
      eprint={2601.14058},
      archivePrefix={arXiv},
      primaryClass={math.DG},
      url={https://arxiv.org/abs/2601.14058}, 
      note={Preprint: \url{https://arxiv.org/abs/2601.14058}}
}

@misc{GRZ26+,
      title={{PDE} aspects of the dynamical optimal transport in the {L}orentzian setting}, 
      author={Gigli, Nicola and Rott, Felix and Zanardini, Matteo},
      year={2026},
      eprint={2601.13167},
      archivePrefix={arXiv},
      primaryClass={math.AP},
      url={https://arxiv.org/abs/2601.13167}, 
      note={Preprint: \url{https://arxiv.org/abs/2601.13167}}
}

@misc{BEOR25+,
      title={Concavity of spacetimes}, 
      author={Beran, Tobias and Er\"os, Darius and Ohta, Shinichi and Rott, Felix},
      year={2025},
      eprint={2509.26196},
      archivePrefix={arXiv},
      primaryClass={math.DG},
      url={https://arxiv.org/abs/2509.26196}, 
      note={Preprint: \url{https://arxiv.org/abs/2509.26196}}
}

@book {CFTbook,
    AUTHOR = {Di Francesco, Philippe and Mathieu, Pierre and S\'en\'echal,
              David},
     TITLE = {Conformal field theory},
    SERIES = {Graduate Texts in Contemporary Physics},
 PUBLISHER = {Springer-Verlag, New York},
      YEAR = {1997},
     PAGES = {xxii+890},
      ISBN = {0-387-94785-X},
   MRCLASS = {81T40 (81-02)},
  MRNUMBER = {1424041},
MRREVIEWER = {Christoph\ Schweigert},
       DOI = {10.1007/978-1-4612-2256-9},
       URL = {https://doi.org/10.1007/978-1-4612-2256-9},
}

@ misc{HBS,
	author = {Gigli, Nicola},
	note = {Preprint: \url{https://arxiv.org/abs/2503.10467}},
	title = {Hyperbolic Banach Spaces},
    year = {2025},
 	eprint={2503.10467},
      	archivePrefix={arXiv},
      	url={https://arxiv.org/abs/2503.10467}
}

@article {Harris,
    AUTHOR = {Harris, Steven G.},
     TITLE = {A triangle comparison theorem for {L}orentz manifolds},
   JOURNAL = {Indiana Univ. Math. J.},
  FJOURNAL = {Indiana University Mathematics Journal},
    VOLUME = {31},
      YEAR = {1982},
    NUMBER = {3},
     PAGES = {289--308},
      ISSN = {0022-2518,1943-5258},
   MRCLASS = {53C50 (53B30)},
  MRNUMBER = {652817},
MRREVIEWER = {J.\ K.\ Beem},
       DOI = {10.1512/iumj.1982.31.31026},
       URL = {https://doi.org/10.1512/iumj.1982.31.31026},
}

@article{mccann2025tradinglinearityellipticitynonsmooth,
      title={Trading linearity for ellipticity: a nonsmooth approach to {E}instein's theory of gravity and the {L}orentzian splitting theorems}, 
      author={Robert J. McCann},
      journal= {Proceedings of the Forward From the Fields Medal 2024, Deirdre Haskell and V. Kumar Murty, eds. (To appear)},
      year={2025},
      note= {\url{https://arxiv.org/abs/2501.00702}}, 
}

@book {One83,
    AUTHOR = {O'Neill, Barrett},
     TITLE = {Semi-{R}iemannian geometry},
    SERIES = {Pure and Applied Mathematics},
    VOLUME = {103},
      NOTE = {With applications to relativity},
 PUBLISHER = {Academic Press, Inc. [Harcourt Brace Jovanovich, Publishers],
              New York},
      YEAR = {1983},
     PAGES = {xiii+468},
      ISBN = {0-12-526740-1},
   MRCLASS = {53-01 (53B30 53C50 83-02)},
  MRNUMBER = {719023},
MRREVIEWER = {N.\ V.\ Mitskevich},
}
\bibliographystyle{abbrv}

\end{document}